\documentclass{amsart}
\usepackage[utf8]{inputenc}
\usepackage{ninecolors} 
\usepackage{fullpage,hyperref,mathtools,verbatim,tikz,amssymb,amsmath,amsthm}
\usepackage{thmtools}
\usepackage[capitalise]{cleveref}
\hypersetup{colorlinks, citecolor=red5, filecolor=black, linkcolor=blue5, urlcolor=blue5}
\usepackage{dsfont}
\usetikzlibrary{calc}
\usepackage{diagbox}
\usepackage{pict2e}
\usepackage{tikz-cd}
\usepackage{float}
\usepackage{shuffle}
\usepackage[shortlabels]{enumitem}
\usepackage{soul}
\usepackage{subcaption}
\usepackage{xr} 

\usepackage[normalem]{ulem} 
\usepackage{soul}

\setlist[itemize]{noitemsep, topsep=0pt}
\setlist[enumerate]{noitemsep, topsep=0pt}
\newcommand{\seqnum}[1]{\href{http://oeis.org/#1}{\underline{#1}}}

\newcommand{\suchthat}{\colon}
\newcommand{\N}{\mathbb{N}}
\newcommand{\Z}{\mathbb{Z}}

\renewcommand{\S}{\mathfrak{S}}
\newcommand{\Sym}{\mathfrak{S}}
\newcommand{\Cat}{\mathrm{Cat}}

\DeclareMathOperator{\des}{des}
\DeclareMathOperator{\Des}{Des}
\DeclareMathOperator{\asc}{asc}
\DeclareMathOperator{\Asc}{Asc}

\DeclareMathOperator{\id}{id}
\DeclareMathOperator{\inv}{inv}
\DeclareMathOperator{\Inv}{Inv}
\DeclareMathOperator{\cont}{cont} 
\DeclareMathOperator{\maj}{maj}

\DeclareMathOperator{\code}{code}

\DeclareMathOperator{\Out}{\mathcal{C}}
\DeclareMathOperator{\IPF}{IPF} 
\DeclareMathOperator{\PF}{PF}
\DeclareMathOperator{\UPF}{UPF} 
\DeclareMathOperator{\upf}{upf} 
\DeclareMathOperator{\UFR}{UFR}
\DeclareMathOperator{\Fub}{Fub}
\DeclareMathOperator{\Fib}{Fib}

\DeclareMathOperator{\type}{type}
\DeclareMathOperator{\spot}{spot}
\DeclareMathOperator{\car}{car}
\DeclareMathOperator{\ones}{ones}

\newcommand{\FF}{\mathcal{F}}
\newcommand{\GG}{\mathcal{G}}
\newcommand{\PPP}{\mathcal{P}}
\newcommand{\RRR}{\mathcal{R}}
\newcommand{\SSS}{\mathcal{S}}
\newcommand{\x}{\times}

\DeclareMathOperator{\disp}{disp} 
\DeclareMathOperator{\maxdisp}{maxdisp} 
\newcommand{\defterm}{\textbf} 

\newcommand{\qbinom}[2]{\genfrac{[}{]}{0pt}{}{#1}{#2}}
\newcommand{\sm}{\setminus}

\newcommand{\lp}{\left(}
\newcommand{\rp}{\right)}

\newtheorem{theorem}{Theorem}[section]
\newtheorem{corollary}[theorem]{Corollary}
\newtheorem{proposition}[theorem]{Proposition}
\newtheorem{problem}[theorem]{Problem}
\newtheorem{lemma}[theorem]{Lemma}
\newtheorem{conjecture}[theorem]{Conjecture}
\theoremstyle{definition}
\newtheorem{definition}[theorem]{Definition}
\newtheorem{remark}[theorem]{Remark}
 \newtheorem{example}[theorem]{Example}

\begin{document}

\begin{abstract}
We study enumerative properties of the set $\IPF_\ell(n)$: parking functions of length $n$ in which each car parks at most $\ell$ spaces past its preferred spot.
First, we prove that $\IPF_1(n)$ and $\IPF_2(n)$ are each invariant under the well-known Foata bijection that interchanges inversions and major index, implying that these statistics are equidistributed on $\IPF_\ell(n)$.
Second, we obtain explicit generating functions for $\IPF_1(n)$ and $\IPF_2(n)$ by inversions and displacement, as well as a cyclic sieving phenomenon for $\IPF_1(n)$.
Third, we obtain formulas for the number of elements of $\IPF_\ell(n)$ with given sum, using a generalization of Lehmer codes and a bijection due to Avalos and Bly. 
\end{abstract}

\subjclass{Primary 05A05; 
Secondary 05A15, 
05A19} 
\keywords{Parking function, $\ell$-interval parking function, unit interval parking function, unit Fubini ranking, Lehmer code, ordered set partition, permutation, Foata bijection, cyclic sieving, $q$-analogue, inversion, major index, displacement}

\title{$\ell$-Interval Parking Functions:\\
Foata invariance, statistic generating functions, and ciphers}

\author[Celano]{Kyle Celano}
\address[K.~Celano]{Department of Mathematics, Wake Forest University, NC 27109}
\email{\url{celanok@wfu.edu}}

\author[Elder]{Jennifer Elder}
\address[J.~Elder]{Department of Computer Science, Mathematics and Physics, Missouri Western State University, St. Joseph, MO 64507}
\email{\textcolor{blue}{\href{mailto:jelder8@missouriwestern.edu}{jelder8@missouriwestern.edu}}} 

\author[Hadaway]{Kimberly P. Hadaway}
\address[K.~P.~Hadaway]{Department of Mathematics, Iowa State University, Ames, IA 50010}
\email{\textcolor{blue}{\href{mailto:kph3@iastate.edu}{kph3@iastate.edu}}}

\author[Harris]{Pamela E. Harris}
\address[P.~E.~Harris]{Department of Mathematical Sciences, University of Wisconsin-Milwaukee, Milwaukee, WI 53211}
\email{\textcolor{blue}{\href{mailto:peharris@uwm.edu}{peharris@uwm.edu}}}
 
\author[Martin]{Jeremy L. Martin}
\address[J.~L.~Martin]{Department of Mathematics, University of Kansas, Lawrence, KS 66045}
\email{\textcolor{blue}{\href{mailto:jlmartin@ku.edu}{jlmartin@ku.edu}}} 

\author[Priestley]{Amanda Priestley}
\address[A.~Priestley]{Department of Mathematics, King's College London, 
London, WC2R 2LS
United Kingdom}
\email{\textcolor{blue}{\href{mailto:amanda.priestley@kcl.ac.uk}{amanda.priestley@kcl.ac.uk}}} 

\author[Udell]{Gabe Udell}
\address[G.~Udell]{Mathematics Department, University of California, Santa Cruz, Santa Cruz, CA  95064}
\email{\textcolor{blue}{\href{mailto:gudell@ucsc.edu}{gudell@ucsc.edu}}}

\date{\today} 
\maketitle

\section{Introduction} 
This paper is about enumerating $\ell$-interval parking functions with respect to statistics such as inversion, displacement, and major index. 
Parking functions first appeared in the work of Konheim and Weiss \cite{Konheim1966} and have become standard objects in combinatorics.  
Suppose that $n$ cars attempt to park in $n$ spots $1,2,\dots,n$ on a one-way street.  
Each car drives to its preferred spot $a_i$ and parks there if possible; if that spot is taken, the car continues along the street and parks in the first available spot.  If all cars are able to park successfully, the tuple $\alpha = (a_1,a_2,\dots,a_n)\in [n]^n$ is called a \defterm{parking function}.  
In this case, if the $i$th car parks in spot $s_i$, then the number $s_i-a_i$ is its \defterm{displacement}.

An \defterm{$\ell$-interval parking function}, or IPF for short, is a parking function in which each car has displacement less than or equal to some fixed number $\ell$; we write $\IPF_n(\ell)$ for the set of all $\ell$-interval parking functions of length~$n$.  
These objects were first studied by 
Aguilar-Fraga et al.~\cite{aguilarfraga2024interval}.  (The idea of restricting each car to an interval was considered by Colaric et al.~\cite{bib:Colaric2020IntervalPF}, although in that paper the intervals were not required to have the same length.)  The case $\ell=1$ is of particular interest.  These \defterm{unit interval parking functions}, or UPFs, are enumerated by the Fubini numbers \cite{bradt2024unit, bib:HadawayUndergradThesis}, which also count the faces of the permutahedron \cite{unit_perm}.  UPFs are also related  to Boolean intervals in weak Bruhat order on the symmetric group \cite{elder2024parking}. 
Enumerative results on $\ell$-interval parking functions with $\ell\geq 1$ also give connections to Dyck paths with restricted heights and to preferential arrangements \cite{aguilarfraga2024interval}.

Bradt et al.~\cite[Theorem~2.9]{bradt2024unit} proved that unit interval parking functions exhibit a tightly controlled \emph{block structure}, which we describe in Section~\ref{sec:block structure}.  
To summarize, if $\alpha=(a_1,\dots,a_n)$ is a UPF with weakly increasing rearrangement $\alpha^\uparrow=(a'_1,\dots,a'_n)$, then $a'_i\in\{i,i-1\}$ for every $i$, and if $\alpha^\uparrow$ is partitioned into blocks by inserting a separator before every entry with $\alpha'_i=i$,
then the elements of each block appear left-to-right in $\alpha$.  This structure is key to many of the 
enumerative results in this paper. 
No analogous structure is known for $\ell$-interval parking functions with $\ell\geq2$.

In \Cref{sec:equidistribution}, we prove that the inversion and major index statistics are equidistributed on $\ell$-interval parking functions when $\ell\leq2$ or $\ell\geq n-2$.  
Recall that the \defterm{major index} of $(w_1,w_2,\dots,w_n)\in\Z^n$ is $\maj(w)=\sum_{i\in \Des(w)}i$, where $\Des(w)=\{i\in [n-1]\suchthat w_i>w_{i+1}\}$. 
Foata \cite{foata} famously constructed a content-preserving bijection~$F$ on the set of words over a fixed alphabet
with the property that $\inv(F(w))=\maj(w)$, thus proving that the two statistics are equidistributed on any set of words invariant under permutation (as proved earlier by MacMahon \cite{MacMahon}).
Parking functions are invariant under permutation, but in general, the set of $\ell$-interval parking functions is not.  Nevertheless, we prove (\Cref{thm:main-Foata}) that Foata's bijection preserves the set $\IPF_n(\ell)$ precisely when $\ell\in\{0,1,2,n-2,n-1\}$. 
The most difficult cases are $\ell=1$ and $\ell=2$.  Along the way, we prove several results (\cref{lemma:rearrange}, \cref{prop:sorting}, \cref{lemma:reserved-spot}) to describe how small changes in a parking function (such as transposing two consecutive entries) affect where the cars park; these facts may be of independent interest and utility.

In Section 4, we study the generating function $\Phi_{n,\ell}(q,t)$ for $\IPF_n(\ell)$ by total displacement and inversion number. In \cite{inversions_in_pfs}, a subset of the authors found $t$-exponential generating functions for $\Phi_{n,n-1}(1,t)$ and $\Phi_{n,1}(1,t)$. In this, we provide succinct formulas $\Phi_{n,\ell}(q,t)$ for $\ell\in \{1,2,n-2\}$ (\Cref{thm:disp-inv-enumerator-for-upfs}, \Cref{thm:disp-enumerator-for-2-interval}, \Cref{ell-minus-two}).  By the Foata invariance of \Cref{sec:equidistribution}, this also give us formulas for the displacement-major index generating functions \Cref{cor:foata-disp-inv}. For unit interval parking functions, if we write $\Phi_{n,1}(q,t)=\sum_k q^k f_{n,k}(t)$, then each polynomial $f_{n,k}(t)$ exhibits a cyclic sieving phenomenon (\cref{thm:cyclic-sieving}). Additionally, we derive a general formula in \cref{thm:disp-enumerator-for-ell-interval} for generating functions of the form $\sum_{\alpha\in \IPF_n(\ell)}q^{\disp(\alpha)}t^{\operatorname{stat}(\car_\alpha)}$ where $\car_\alpha$ is the permutation $\car_\alpha(i)$ denotes the car that parks in spot $i$ and $\operatorname{stat}:\S_n\to \mathbb{C}$. When $\ell=1$ and $\operatorname{stat}$ is the number of inversions, this is the same as $\Phi_{n,1}$ and reduces to the relatively simple expression in \cref{thm:disp-inv-enumerator-for-upfs}.
The case $\ell=n-1$ was previously obtained in \cite{bib:ColmenarejoHarris}.

In \Cref{sec:upf-inversions}, we present a method for calculating  $\upf_{n,k}^{\inv}$, the number of UPFs of length $n$ with exactly $k$ inversions. The main tool is a collection of new combinatorial objects called ``ciphers''. These are in bijection with UPFs and appeared implicitly in work of Avalos and Bly~\cite{avalos_sequences_2020}.  We work out explicit formulas for $k\leq3$ (Equations \eqref{upf:n1}--\eqref{upf:n3}).  The formulas suggest that $\upf_{n,k}^{\inv}=O(2^{n-2k} n^k)$ in general (\cref{conj:count-upfnk}).  As another application of ciphers, we recover the main result of \cite{elder2024parking}, a bijection between
unit Fubini rankings (UPFs with all blocks of sizes 1 or 2) and Boolean intervals in weak order (\cref{count-unit-Fubini}).

We conclude in \Cref{sec:future-work} with some open problems for future study.

\section{Background} \label{sec: definitions}

In this section, we provide background and definitions related to parking functions and $\ell$-interval parking functions.
A standard source for the basics of parking functions is \cite{YanPFs}.
For interval and $\ell$-interval parking functions, see, e.g., \cite{aguilarfraga2024interval,bradt2024unit,bib:Colaric2020IntervalPF}.
\begin{remark}
    Throughout we interchangeably regard tuples $(a_1,a_2,\dots,a_n)$ and words $a_1a_2\cdots a_n$ as identical objects, depending on context. 
\end{remark}

Throughout, the symbol $\N$  denotes the set of positive integers.  
For $m,n\in\N$, we set $[n]=\{1,2,\dots,n\}$ and $[m,n]=\{m,m+1,\dots,n\}$.  
For a tuple $\alpha=(a_1,a_2,\dots,a_n)\in\N^n$, we write $\alpha^\uparrow$ for the tuple obtained by sorting the elements of $\alpha$ in weakly increasing order. The symbol $\mathfrak{S}_n$ denotes the symmetric group of permutations of~$[n]$.  Throughout, we write permutations in one-line notation.

\subsection{Parking functions and \texorpdfstring{$\ell$-interval}{l-interval} parking functions}\label{sec:background on PFs}
We write $\PF_n$ for the set of parking functions of length $n$. The number of parking functions of length $n$ is $|\PF_n|=(n+1)^{n-1}$~\cite{Konheim1966}.  
Moreover,
a tuple $\alpha=(a_1,a_2,\dots,a_n)\in[n]^n$ is a parking function if and only if its weakly increasing rearrangement  $\alpha^\uparrow=(a'_1,a'_2,\dots,a'_n)$ satisfies $a'_i\leq i$ for all $i$, or equivalently if there is some permutation $\sigma\in\Sym_n$ such that $a_i\leq\sigma_i$ for all $i\in[n]$ \cite[p.~4]{YanPFs}.   
It follows that if $\alpha\in\PF_n$ and $\beta=(b_1,b_2,\dots,b_n)$ satisfy $1\leq b_i\leq a_i$ for all $i\in[n]$, then $\beta\in\PF_n$. We refer to this property as the \defterm{rearrangement criterion}.

We write $\PF_n^\uparrow$ for the set of weakly increasing parking functions of length $n$. The number of weakly increasing parking functions of length is $|\PF_n^\uparrow|=\Cat_n=\frac{1}{n+1}\binom{2n}{n}$, the $n$th Catalan number~\cite[Exercise~6.19(s)]{stanley1999enumerative}.

\begin{definition}\label{def:ell_interval}
Let $\alpha = (a_1,a_2, \ldots , a_n)\in \PF_n$.
The \defterm{car permutation} of $\alpha$ is $\car_\alpha=(\car_\alpha(1),\dots,\car_\alpha(n))\in\Sym_n$, where $\car_\alpha(i)$ denotes the car that parks in the $i$th spot.
The \defterm{spot permutation} of $\alpha$ is $\spot_\alpha=(\spot_\alpha(1),\dots,\spot_\alpha(n))\in\Sym_n$, where $\spot_\alpha(i)$ denotes the number of spot that the $i$th car parks in.
Note $\car_\alpha$ and $\spot_\alpha$ are permutations written as tuples and  $\car_\alpha=\spot_\alpha^{-1}$.
\end{definition}

\begin{example}\label{ex to start}
If $\alpha=(1,4,4,3,2,2)\in\PF_n$, then $\car_\alpha=(1,5,4,2,3,6)$ and $\spot_\alpha=(1,4,5,3,2,6)$.
\end{example}

\begin{remark}
Much of the parking function literature uses the term ``outcome'' for what we have called the car permutation of a parking function $\alpha$.  We have adopted the terms ``car permutation'' and ``spot permutation'' because we work with both of them in different parts of this paper, and wish to avoid any possible ambiguity in the use of the word ``outcome''.  In the Sage ParkingFunctions library, the car permutation is computed as 
\texttt{alpha.cars\_permutation()} and the spot permutation is \texttt{alpha.parking\_permutation()}.
\end{remark}

\begin{definition}
Let $\alpha=(a_1,a_2,\ldots,a_n)\in\PF_n$.
The \defterm{displacement} of the $i$th car is $\disp_\alpha(i)=\spot_\alpha(i)-a_i$.  That is, $\disp_{\alpha}(i)$ is the number of additional spaces car $i$ has to drive past its preferred spot in order to park.
The \defterm{(total) displacement}\footnote{Total displacement is also called \defterm{area}, since it is exactly the area under the labeled Dyck path representing $\alpha$; see, e.g., \cite[pp.\ 54--55]{YanPFs}.} of $\alpha$ is $\disp(\alpha)=\sum_{i=1}^n \disp_{\alpha}(i)$.
The \defterm{maximum displacement} of $\alpha$ is $\maxdisp(\alpha)=\max(\disp_\alpha(1),\disp_\alpha(2),\dots,\disp_\alpha(n))$.
\end{definition}

\begin{example}[Continuing \Cref{ex to start}]
If $\alpha=(1,4,4,3,2,2)\in\PF_n$, then 
the displacements of cars $1,2,\ldots,6$ are 
$0,0,1,0,0,4$, respectively.
\end{example}

\begin{definition}
For a positive integer $\ell\leq n$, we say that $\alpha$ is an \defterm{$\ell$-interval parking function} if $\maxdisp(\alpha)\leq\ell$.
The set of $\ell$-interval parking functions of length $n$ is denoted $\IPF_n(\ell)$.
In the case $\ell=1$, we call $\alpha$ a \defterm{unit interval parking function} and write $\UPF_n=\IPF_n(1)$.
\end{definition}

\begin{example} \label{ex:basic counts}
We can give exact formulas for $|\IPF_n(\ell)|$ in some extreme cases.
\medskip

\begin{itemize}[leftmargin=.7in]

\item[\underline{$\ell=n-1$:}] Every $\alpha\in\PF_n$ satisfies $\maxdisp(\alpha)\leq n-1$, so
\[|\IPF_n(n-1)|=|\PF_n|=(n+1)^{n-1}.\]

\item[\underline{$\ell=n-2$:}] Let $\alpha=(a_1,a_2,\ldots,a_n)\in\PF_n$.  
Then $\maxdisp(\alpha)=n-1$ if and only if the car that parks in spot~$n$ preferred spot~1.  
This can happen only if (i) that car was the last to park and (ii) the first $n-1$ cars parked in the first $n-1$ spots.  Equivalently, $a_n=1$ and $(a_1,a_2,\dots,a_{n-1})\in\PF_{n-1}$.  Therefore,
\begin{align*}
|\IPF_n(n-2)|&=|\PF_n|-|\{\alpha\in\PF_n\suchthat \maxdisp(\alpha)=n-1\}|=(n+1)^{n-1}-n^{n-2}.
\end{align*}

\item[\underline{$\ell=1$:}] The parking functions $\alpha$ with $\maxdisp(\alpha)=1$ are precisely the unit interval parking functions.  Therefore,
\[|\IPF_n(1)|=|\UPF_n|=\Fub_n,\]
where $\Fub_n$ is the number of ordered set partitions of~$[n]$ (see \cref{rmk:fubini} below).

\item[\underline{$\ell=0$:}] The parking functions $\alpha$ with $\maxdisp(\alpha)=0$ are precisely the permutations.  Therefore,
\[|\IPF_n(0)|=n!.\]

\end{itemize}
\end{example}

It is immediate from the rearrangement criterion that any rearrangement of a parking function is a parking function.
The same is not true for $\ell$-parking functions because maximum displacement may change under rearrangement. 
For instance, $\maxdisp(112)=1$ but $\maxdisp(121)=2$.
On the other hand, a weaker result holds for all~$\ell$, as we prove later (\cref{prop:sorting}).

\subsection{Block structure of unit interval parking functions}\label{sec:block structure}

In \cite{bradt2024unit}, it was proved that unit interval parking functions are characterized by a tightly controlled \defterm{block structure}, as we now explain.

\begin{definition}
A tuple $\beta=(b_1,b_2,\dots,b_n)\in[n]^n$ is called a \defterm{block word} if $b_i\in\{i,i-1\}$ for each $i$ (so in particular, $b_1=1$).
When we write a block word, we place a separator before each $b_i$ such that $b_i=i>1$.
The maximal subwords between separators are called \defterm{blocks}. 
Note that a block word is determined by the ordered list of lengths of its blocks.
\end{definition}

For example, the block words of length 3 are $112$, $11\mid3$, $1\mid22$, $1\mid2\mid3$.

We now reformulate the characterization of unit interval parking functions based on their block structure.

\begin{theorem}[{\cite[Theorem~2.9]{bradt2024unit}}]
\label{thm:upf_rearrangement}
Let $\alpha = (a_1, a_2,\dots, a_n) \in \UPF_{n}$ and $\alpha^\uparrow = (a'_1, a'_2,\dots, a'_n)$.
Then,
\begin{enumerate}
\item The increasing rearrangement $\alpha^{\uparrow}$ is a block word. If $\alpha^\uparrow=\pi_1\mid\pi_2\mid\cdots\mid\pi_m$, we refer to the collection of $\pi_j$'s as the \defterm{block structure} of $\alpha$.
\item The entries in each block appear in increasing order in $\alpha$. That is, $\alpha$ is a shuffle of its blocks.
\item Given a block word $\beta$, every shuffle of its blocks is a unit interval parking function.  Therefore,
\[|\{\alpha\in\UPF_n \suchthat \alpha^\uparrow=\beta\}| = \binom{n}{|\pi_1|,|\pi_2|,\ldots, |\pi_m|}.\]
\end{enumerate}
\end{theorem}

Next we state a bijection between the set of unit interval parking functions of length $n$ with~$m$ blocks and the set of surjective functions from $[n]$ to $[m]$.

\begin{corollary}\label{cor:block-bijection}
Let $\alpha=(a_1,a_2,\dots,a_n)\in\UPF_n$ have block structure $\pi_1\mid\pi_2\mid\cdots\mid\pi_m$.  For each $i\in[n]$, let $s(i)$ be the index of the block to which $a_i$ belongs.  Thus, $s$ is a surjective function $[n]\to[m]$ (equivalently, an ordered partition of $[n]$ into $m$ nonempty blocks).  Then the map sending $\alpha\mapsto s$ is a bijection
\[\{\alpha\in\UPF_n\suchthat \alpha \text{ has $m$ blocks}\} \to \{\text{surjective functions } [n]\to[m]\}.\]
\end{corollary}

\begin{remark}\label{rmk:fubini}
A surjective function $[n]\to[m]$ can be regarded as an ordered set partition of $[n]$ with $m$ nonempty blocks.  In particular, the number of unit interval parking functions of length $n$ equals the \defterm{Fubini number} $\Fub_n$ of ordered set partitions of $[n]$, as observed by Hadaway \cite{bib:HadawayUndergradThesis}.  The Fubini numbers are sequence \#\seqnum{A000670} in \cite{OEIS}. We obtain a new proof this fact following \Cref{thm:disp-inv-enumerator-for-upfs}.
\end{remark}

\begin{example}\label{ex: block structure}
Let $\alpha=(8,1,5,5,1,2,4,7)\in\UPF_{8}$.  Then the block structure  is $\alpha^\uparrow=112\mid4\mid55\mid7\mid8$, and the corresponding element of $[5]^8$, as in \cref{cor:block-bijection}, is $s=51331124$.
\end{example}

Block structure appears to be unique to the case $\ell=1$; no result analogous to \cref{thm:upf_rearrangement} is known for $\ell$-interval parking functions when $\ell>1$. We state this as an open problem in \Cref{sec:future-work}.

\subsection{Permutation statistics}\label{sec:background on perm stats}
Let $w=w_1\cdots w_n\in \N^n$ be a sequence of positive integers.  An \defterm{inversion} of $w$ is a pair of indices $(i,j)$ such that $1\leq i<j\leq n$ and $w_i>w_j$. 
We let $\Inv(w)$ be the set of inversions of $w$ and let $\inv(w)=|\Inv(w)|$.  
A \defterm{descent} is a position $i\in [n-1]$ such that $w_i>w_{i+1}$.  
We let $\Des(w)$ be the set of descents of $w$ and let $\des(w)=|\Des(w)|$.  
The \defterm{major index} of $w$ is
\[\maj(w)=\sum_{i\in \Des(w)}i.\]
By a bijection of Foata \cite{MacMahon,foata1974mappings} that we describe in detail in \Cref{sec:equidistribution}, inversions and major index are \defterm{equidistributed} as statistics on the set of permutations, meaning that there are as many permutations $w\in \S_n$ with $\inv(w)=j$ as there are with $\maj(w)=j$, for each nonnegative integer $j$. 
More generally, inversions and major index are equidistributed on any set of words $W\subseteq \N^n$ satisfying the property that if $w\in W$ and $u$ is a rearrangement of $w$, then $u\in W$ \cite{Foata1978}. 
Hence, inversions and major index are equidistributed on the set of parking functions. In \Cref{sec:equidistribution}, we completely characterize the pairs $(n,\ell)$ such that inversions and major index are equidistributed on $\IPF_n(\ell)$.

It is immediate from the definition that $\inv(\sigma)=\inv(\sigma^{-1})$ for every permutation $\sigma$.
For a surjective function $s:[n]\to [m]$, let $\Inv(s)=\Inv(w)$, where $w=w_1\cdots w_n$ is defined as $w_i=s(i)$ for each $\in [n]$. The following result asserts that inversions of unit interval parking functions can be computed as inversions of permutations or surjective functions.

\begin{lemma}\label{lem: inv of outcome}
Let $\alpha\in \UPF_n$ with corresponding surjective function $s$.  Then $\Inv(\alpha)=\Inv(\spot_\alpha)=\Inv(s)$.
\end{lemma}

\begin{proof}
Let $\alpha=(a_1,a_2,\ldots,a_n)\in\UPF_n$ and let $(i,j)$ be an inversion of $\spot_\alpha$. 
That is, car~$i$ arrives earlier than car~$j$ and parks to the the right of car $j$.  It follows that $a_i>a_j$, so $(i,j)\in\Inv(\alpha)$.

Now let $(i,j)$ be an inversion of~$\alpha$. That is, $a_i>a_j$ and $i<j$. Since $\alpha$ is a UPF, $\spot_\alpha(j)\leq a_j+1\leq a_i \leq \spot_\alpha(i)$.  
Thus,every inversion of~$\alpha$ is an inversion of~$\spot_\alpha$.

For the second equality, by Part (b) of \Cref{thm:upf_rearrangement}, we have that for all $i<j$, $a_i>a_j$ if and only if $s(i)>s(j)$. Hence, the claim follows.
\end{proof}

An \defterm{ascent} of a word $w\in \N^n$ is a position $i\in [n-1]$ such that $w_i<w_{i+1}$. We let $\Asc(w)$ be the set of ascents of $w$ and set $\asc(w)=|\Asc(w)|$. It is known \cite[Exercise 1.133a]{stanley2012enumerative} that
\begin{equation}\label{eq:ascent-power-of-2-are-fubini}
    \sum_{\sigma\in \S_n} 2^{\asc(\sigma)}=\Fub_n=|\UPF_n|.
\end{equation}  We obtain a more refined enumerative result in \Cref{thm:disp-inv-enumerator-for-upfs}.

\section{Foata invariance and inv-maj equidistribution}\label{sec:equidistribution}
In this section, we prove that a classical bijection of Foata \cite{foata} preserves the classes of $\ell$-interval parking functions for $\ell\in\{1,2,n-2,n-1\}$, from which it follows that
the inversion and major index statistics are equidistributed on each of these classes.

\begin{theorem} \label{thm:main-Foata}
Let $1\leq\ell<n$, and let $\IPF_n(\ell)$ denote the class of $\ell$-interval parking functions of length $n$.
\begin{itemize}
\item If $\ell\in\{1,2,n-2,n-1\}$, then the Foata transform restricts to a bijection $\IPF_n(\ell)\to\IPF_n(\ell)$, and consequently
the inversion and major index statistics are equidistributed on $\IPF_n(\ell)$, i.e.,
\[
\sum_{p\in \IPF_n(\ell)}t^{\inv(p)}= \sum_{p\in \IPF_n(\ell)}t^{\maj(p)}.
\]
\item If $\ell\not\in\{1,2,n-2,n-1\}$, then inversion and major index are not equidistributed on $\IPF_n(\ell)$, so the Foata transform cannot be a bijection $\IPF_n(\ell)\to\IPF_n(\ell)$.
\end{itemize}
\end{theorem}

The proof spans several cases, which we order by increasing difficulty.
\begin{itemize}
\item $\ell=n-2$: \Cref{foata:ell-is-n-minus-two}
\item $3\leq\ell\leq n-3$: \Cref{inv-maj-not-equi} \item $\ell=1$: \Cref{Foata:ell-is-one}
\item $\ell=2$: \Cref{Foata:ell-is-two}
\end{itemize}

\subsection{The Foata transform on words}
We begin by describing Foata's bijection \cite{foata} and its relevant properties.

\begin{definition}[\cite{foata}]\label{defn:foata}
The \defterm{Foata transform} is the function~$F$ defined inductively on words $w=w_1\cdots w_n$ as follows.
\begin{itemize}
\item For $n=1$, we set $F(w)=w$.
\item For $n>1$, the image $F(w)$ is defined as follows.  Let $w'=F(w_1\cdots w_{n-1})=w'_1\cdots w'_{n-1}$. 
\begin{itemize}
\item Place a separator before $w_1$
\item If $w_n\geq w'_{n-1}$, then place a separator after every $w'_i$ such that $w_n\geq w'_i$.
\item If $w_n<w'_{n-1}$, then place a separator after every $w'_i$ such that $w_n<w'_i$.
\end{itemize}
\item Cycle each \defterm{segment} (i.e., each maximal subword between consecutive separators) by moving its rightmost entry to the start.
\item Finally, append $w_n$.  The result is $F(w)$.
\end{itemize}
\end{definition}

\begin{example}
    We show how to compute $F(w)$ for $w=315312$. Each line is showing how you compute $F(w_1\cdots w_i)$ by first placing the separators to form the segments, and then the cycling segments and placing $w_i$ on the end.
    \[\begin{array}{lclc}
        3&\\
        |3|&\rightarrow&31\\
        |3|1|&\rightarrow&315\\
        |315|&\rightarrow&5313\\
        |5|3|13|&\rightarrow&53311\\
        |5331|1|&\rightarrow&153312&=F(w)
    \end{array}\]
\end{example}

Foata \cite{foata} proved that $F$ is a bijection such that
$\inv(F(w))=\maj(w)$ for all words $w$. This implies that the inversion number and major index are equidistributed for permutations (which are $0$-interval parking functions), i.e.,
\begin{equation} \label{equi:Sn}
\sum_{w\in\Sym_n} q^{\inv(w)} = \sum_{w\in\Sym_n} q^{\maj(w)}.
\end{equation}
Equidistribution had been proved previously by MacMahon \cite{MacMahon}, but the Foata transform gave the first combinatorial proof.  For more detail, see \cite[pp.~41--43]{stanley2012enumerative}.

It is immediate from the definition of $F$ that it preserves the content of a word, i.e., the number of occurrences of each letter.  In particular, it preserves the class of parking functions (which are $(n-1)$-interval parking functions). Moreover, the definition of $F$ also implies that for each $i$, the ``partial Foata transform'' $F_i$ defined by $F_i(w) = F(w_1\cdots w_i) w_{i+1} \cdots w_n$ is also a bijection.

The motivating question of this section is: \textit{When does the Foata transform preserve the class of $\ell$-interval parking functions of length~$n$?}  When the answer is positive, equidistribution of inversion number and major index for that class follows immediately.

\subsection{Permuting parking functions}

In this section, we need to compare different parking functions related by permutations (specifically, by partial Foata transforms) and keep track of the effects on individual cars. Accordingly, we assign each car a permanent unique identifier (its ``license plate'', so to speak) and denote a parking function $p=(p_1,\dots,p_n)$ by a \textit{two-line} array
\[p=\left(\begin{array}{ccc}
C_1&\cdots&C_n\\ p_1&\cdots&p_n\end{array}\right)\]
in which $C_i$ means the car with license plate $i$; $p_i$ is its preference; and the order of columns indicates the order with which cars enter the parking lot.

For example, let $p$ be the parking function 112, in which the $i$th car to enter has license plate~$i$, and let $q$ be the parking function obtained by interchanging the second and third cars.  In our notation, we write
\[p=\left(\begin{array}{ccc} C_1&C_2&C_3\\ 1&1&2\end{array}\right),\qquad
 q=\left(\begin{array}{ccc} C_1&C_3&C_2\\ 1&2&1\end{array}\right),\]
from which it is easy to compare the outcomes of $p$ and $q$.

We write $\spot_p(C_i)$ for the spot in which car $C_i$ parks with respect to $p$, and define
$\disp_p(C_i)=\spot_p(C_i)-p_i$.  For example, 
if $p$ is the parking function
\[p=\left(\begin{array}{cccccc} C_1&C_2&C_3&C_4&C_5&C_6\\ 1&1&1&1&4&3\end{array}\right)\]
then its Foata transform is
\[r=F(p)=\left(\begin{array}{cccccc} C_5&C_1&C_2&C_3&C_4&C_6\\ 4&1&1&1&1&3\end{array}\right).\]
We can see for example that $\spot_r(C_4)=5$ and $\disp_r(C_4)=5-1=4$. 
Recall that a parking function $p$ is an $\ell$-interval parking function if and only if $\max_i\{\disp_p(C_i)\}\leq \ell$.

\begin{lemma}[Interval Rearrangement]  \label{lemma:rearrange}
Let $p$ be a parking function of length $n$.
Partition $[n]$ into intervals $B_1,\dots,B_k$, and let $q=(p_{w^{-1}(1)},\dots,p_{w^{-1}(n)})$ for some $w\in\Sym_{B_1}\times\cdots
\times\Sym_{B_k}$.
Then for each $i\in[k]$, we have
\[\{\spot_p(j)\suchthat j\in B_i\}=\{\spot_q(j)\suchthat j\in B_i\}.\]
\end{lemma}
This fact may be known, but we have not found it in the literature, so we give a short proof.

\begin{proof} 
Let
\[p=\left(\begin{array}{ccc}
C_1&\cdots&C_n\\ p_1&\cdots&p_n\end{array}\right).\]
The permutation $w$ can be factored as a product of adjacent transpositions, each one belonging to some $\Sym_{B_j}$, so it suffices to show that interchanging the order of two adjacent cars either preserves or swaps their outcomes. Let 
\[q=\left(\begin{array}{ccccccc}
C_1&\cdots&C_{i-1}&C_{i+1}&C_i&C_{i+2}\cdots&C_n\\ p_1&\cdots&p_{i-1}&p_{i+1}&p_i&p_{i+2}\cdots&p_n\end{array}\right)\]
for some $i$ and suppose that $S$ is the set of parking spots available after cars $C_1,\dots,C_{i-1}$ park in $p$. Set $p_{i}=u$ and $p_{i+1}=v$. Then, we have  
\[a:=\spot_{p}(C_i)=\min\{s\in S \suchthat s\geq u\},\quad  b:=\spot_{p}(C_{i+1})=\min\{s\in S\setminus a \suchthat s\geq v\}\]
and if $q=(p_1,\dots,p_{i+1},p_i,\dots,p_n)$, we have
\[y:=\spot_q(C_{i+1})=\min\{s\in S \suchthat s\geq v\},\quad  z:=\spot_{q}(C_{i})=\min\{s\in S\setminus y \suchthat s\geq u\}.\]
If $u=v$ then evidently $a=y$ and $b=z$.  
Without loss of generality, assume $u<v$.

Evidently $u\leq a\leq z$ and $v\leq y\leq b$.  If $(a,y)=(z,b)$ then we are done. 
Otherwise, either $a<z$ or $y<b$.  In the first case, the definitions of $a$ and $z$ imply that $y=a$, from which it follows that $u,u+1,\dots,v-1\not\in S$, and these two conditions together imply $z=b$.
A similar argument shows that $y<b$ implies $(y,z)=(a,b)$.  
\end{proof}

\Cref{lemma:rearrange} generalizes easily from parking functions to parking completions (see \cite{bib:ParkingCompletions}).  One useful consequence is the following.

\begin{proposition} \label{prop:sorting}
Suppose $p,q$ are two parking functions obtained by switching two adjacent cars, as follows:
\[p=\lp\begin{array}{cccc}
    \cdots&C_1&C_2 &\cdots\\
    \cdots&x_1&x_2 &\cdots
\end{array}\rp, \qquad
q=\lp\begin{array}{cccc}
    \cdots&C_2&C_1 &\cdots\\
    \cdots&x_2&x_1 &\cdots
\end{array}\rp,\]
with $x_1>x_2$.
Then $\maxdisp(p)\geq\maxdisp(q)$.  In particular, the increasing rearrangement of every $\ell$-interval parking function is an $\ell$-interval parking function.
\end{proposition}

\begin{proof}
By interval rearrangement 
(\Cref{lemma:rearrange}), the sets $\{\spot_p(C_1),\spot_p(C_2)\}$ and $\{\spot_q(C_1),\spot_q(C_2)\}$ are equal, say to $\{a,b\}$, where $a<b$.
If $\spot_p(C_1)=\spot_q(C_1)$ and $\spot_p(C_2)=\spot_q(C_2)$, then $\maxdisp(p)=\maxdisp(q)$.  Otherwise, we have without loss of generality
\[a = \spot_p(C_1)=\spot_q(C_2) < \spot_p(C_2)=\spot_q(C_1) = b.\]
In particular, $b>a\geq x_1>x_2$, and
\begin{align*}
 &\max(\disp_p(C_1),\disp_p(C_2)) = \max(a-x_1,b-x_2)\\
>&\max(\disp_q(C_1),\disp_q(C_2)) = \max(b-x_1,a-x_2).
\end{align*}
Thus, replacing $p$ with $q$ can only decrease the maximum displacement, as claimed.   Iterating this argument shows that bubble-sorting (taking a $\ell$-interval parking function and sorting the smallest element to position 1, then the second smallest to position 2, etc.) produces an  $\ell$-interval parking function.
\end{proof}

\begin{lemma}[Reserved Spot Lemma]\label{lemma:reserved-spot}
For $1\leq j<k\leq n$, let
\[p=\left(\begin{array}{ccc}
C_1&\cdots&C_n\\ p_1&\cdots&p_n\end{array}\right)
\quad\text{and}\quad
q=\left(\begin{array}{ccc|cccc|ccc}
C_1&\cdots&C_{j-1} & C_k&C_j&\cdots&C_{k-1} & C_{k+1}&\cdots&C_n\\
p_1&\cdots&p_{j-1} & p_k&p_j&\cdots&p_{k-1} & p_{k+1}&\cdots&p_n \end{array}\right)\]
be parking functions (where the bars are included merely for clarity), and suppose that $\spot_p(C_k)=\spot_q(C_k)$.  Then $\spot_p(C_i)=\spot_q(C_i)$ for all $i\in[n]$.
\end{lemma}
\begin{proof}
The name of the lemma is also its proof: imagine that in $p$, right after car $C_{j-1}$ parks, driver $C_k$ runs up and puts a cone in spot $p_k$ to reserve it (even if it is not their preferred spot). Since car $C_k$ got to park in spot $p_k$ anyway, this has no impact on the outcome of $p$. On the other hand, ``reserving'' spot $p_k$ after car $C_{j-1}$ parks is equivalent to letting car $C_k$ park there, so the outcome of doing so must be the same as the outcome of $q$. \end{proof}

\subsection{Proof of Theorem~\ref{thm:main-Foata}}

We begin with the easiest cases, showing that the Foata transform is a bijection when $\ell=n-2$, but not when $3\leq\ell\leq n-3$.

\begin{proposition}\label{foata:ell-is-n-minus-two}
For all $n\geq 3$, the Foata transform restricts to a bijection $\IPF_n(n-2)\to\IPF_n(n-2)$.
\end{proposition}
\begin{proof}
A parking function $p$ of length $n$ fails to be an $(n-2)$-parking function precisely when some car has displacement $n-1$.  That car must be the last to park, must prefer spot 1, and must park in spot $n$, which implies that $\hat p=(p_1,\dots,p_{n-1})$ is a parking function of length $n-1$.  
Since $p_n=1$, it follows that $F$ fixes the last digit of $p$ and permutes the entries of $\hat p$.  It follows that $F$ restricts to a bijection $\PF_n\setminus\IPF_n(n-2)\to\PF_n\setminus\IPF_n(n-2)$, hence also to a bijection $\IPF_n(n-2)\to\IPF_n(n-2)$.
\end{proof}

\begin{proposition} \label{inv-maj-not-equi}
    For $n\geq 6$ and $3\leq \ell\leq n-3$, the statistics $\inv$ and $\maj$ are not equidistributed on $\IPF_{n}(\ell)$.
\end{proposition}

\begin{proof}
Let $\mathcal{I}_{n,\ell}=\{p\in \IPF_n(\ell)\suchthat \inv(p)=1\}$ and $\mathcal{M}_{n,\ell}=\{p\in \IPF_n(\ell)\suchthat \maj(p)=1\}$. We show that $|\mathcal{M}_{n,\ell}|<|\mathcal{I}_{n,\ell}|$.  To do so, define
\begin{align*}
\mathcal{P}_{n,\ell}^I &= \left\{(\alpha,i)\in \IPF_n^{\uparrow}(\ell)\times [n-1]\suchthat
	\begin{array}{l}
	\textbf{(I1)}\ \ a_i<a_{i+1}, \text{ and }\\
	\textbf{(I2)}\ \ \text{either $a_{i+1}>i$ or $\disp_{\alpha}(C_i)\leq\ell-1$}
	\end{array} \right\},\\
\mathcal{P}_{n,\ell}^M &= \left\{(\alpha,i)\in \IPF_n^{\uparrow}(\ell)\times [n-1]\suchthat
	\begin{array}{l}
	\textbf{(M1)}\ \ a_i<a_{i+1}, \text{ and }\\
	\textbf{(M2)}\ \ \text{for all $j\in [n]$, if $a_{i+1}\leq j\leq i$, then $\disp_{\alpha}(C_j)\leq \ell-1$}
	\end{array} \right\}.
\end{align*}

Our plan is to define bijections $\kappa:\mathcal{P}_{n,\ell}^I\to\mathcal{I}_{n,\ell}$ and $\lambda:\mathcal{P}_{n,\ell}^M\to\mathcal{M}_{n,\ell}$, then to show that $\mathcal{P}_{n,\ell}^I\supsetneq\mathcal{P}_{n,\ell}^M$.
\medskip

First, given a pair $(\alpha,i)\in\mathcal{P}_{n,\ell}^I$, let $\beta=\kappa(\alpha,i)$ be the parking function obtained by swapping $a_i$ and $a_{i+1}$.  Since $\alpha=\alpha^\uparrow$, it follows that $\beta$ is a parking function with one inversion, and $\alpha$ can be recovered from $\beta$ as its increasing rearrangement.  

By interval rearrangement (\Cref{lemma:rearrange}), every car has the same outcome in $\beta$ as it does in $\alpha$, except that $C_i$ and $C_{i+1}$ may switch places.  Since $\alpha=\alpha^\uparrow$, the two parking spots in question are $i$ and $i+1$.  There are two possibilities:
\begin{itemize}
\item If $a_{i+1}=i+1$, then $\spot_\beta(C_j)=\spot_\alpha(C_j)$ for every $j\in[n]$, and it follows that $\beta\in\IPF_n(\ell)$.
\item If $a_{i+1}\leq i$, then $(\spot_{\beta}(C_i),\spot_{\beta}(C_{i+1}))=(\spot_{\alpha}(C_{i+1}),\spot_{\alpha}(C_i))=(i+1,i)$.  By condition \textbf{(I2)},
$\disp_\beta(C_{i+1})=\disp_\alpha(C_{i+1})-1$ and
$\disp_\beta(C_i)=\disp_\alpha(C_i)+1$.  Hence, we have $\beta\in\IPF_n(\ell)$ if and only if $\disp_\alpha(C_i)\leq\ell-1$.
\end{itemize}
Thus, we have a injective map $\kappa:\mathcal{P}_{n,\ell}^I\to\mathcal{I}_{n,\ell}$.  For surjectivity, observe that if $\beta\in\mathcal{I}_{n,\ell}$ has a unique inversion $ b_i> b_{i+1}$, then $\alpha=\beta^\uparrow\in\IPF_n^\uparrow(\ell)$ by \Cref{prop:sorting}, and the same case argument shows that $(\alpha,i)\in\mathcal{P}_{n,\ell}^I$.
\medskip

Second, given a pair $(\alpha,i)\in\mathcal{P}_{n,\ell}^M$, let $\beta=\lambda(\alpha,i)$ be the parking function obtained by moving $a_{i+1}$ to the front.  Similarly to the  first case, $\beta$ is a parking function with major index~1, and $\alpha=\beta^\uparrow$.

By interval rearrangement (\Cref{lemma:rearrange}), the spots 
and displacements of cars $C_{i+2},\dots,C_n$ are unchanged from $\alpha$ to $\beta$. By construction, we have for all $j \in [n]$,
\[\spot_\beta(C_j)=\begin{cases}
    a_{i+1}&\text{ if }j=i+1,\\
    \spot_\alpha(C_j)& \text{ if }j\leq i \text{ and } \spot_\alpha(C_j)<a_{i+1},\\
    \spot_\alpha(C_j)+1&\text{ if }j\leq i \text{ and } \spot_\alpha(C_j)\geq a_{i+1},
\end{cases}
\]
and
\[\disp_\beta(C_j)=\begin{cases}
    0&\text{ if }j=i+1,\\
    \disp_\alpha(C_j)&\text{ if }j\leq i \text{ and } \spot_\alpha(C_j)<a_{i+1},\\
    \disp_\alpha(C_j)+1&\text{ if }j\leq i \text{ and } \spot_\alpha(C_j)\geq a_{i+1},
\end{cases}
=\begin{cases}
    0&\text{ if }j=i+1,\\
    \disp_\alpha(C_j)&\text{ if }j<a_{i+1},\\
    \disp_\alpha(C_j)+1&\text{ if }a_{i+1}\leq j\leq i.
\end{cases}\]
From this and condition \textbf{(M2)} in the definition of $\mathcal{P}_{n,\ell}^M$, we can conclude that $\beta$ has maximum displacement at most $\ell$, hence belongs to $\mathcal{M}_{n,\ell}$.

Thus, we have an injective map $\lambda:
\mathcal{P}_{n,\ell}^M\to\mathcal{M}_{n,\ell}$.  To see that $\lambda$ is a surjection,
let $\beta\in \mathcal{M}_{n,\ell}$. Then $\alpha=\beta^\uparrow$ has the form
\[\alpha=\lp\begin{array}{ccccccccc}
     C_{2}&C_3 &\cdots& C_{i+1}&C_{1}&C_{i+2} &\cdots& C_n\\
      b_{2}& b_3 &\cdots&  b_{i+1}& b_{1} & b_{i+2}&\cdots& b_n 
\end{array}\rp,\]
where $i$ is the greatest index such that $ b_1> b_{i+1}$.  Then $\alpha\in\IPF_n^\uparrow(\ell)$ by \Cref{prop:sorting}, and again the calculation above shows that $(\alpha,i)\in\mathcal{P}_{n,\ell}^M$.

Third, it is immediate from their definitions that $\mathcal{P}_{n,\ell}^I\supseteq\mathcal{P}_{n,\ell}^M$.
To prove that containment is strict, consider the parking function
\[p=(\underbrace{1,\ldots,1}_{\ell+1 \mbox{ copies}},\ell+1,\ell,\ell+4,\ell+5,\dots,n)\in\mathcal{I}_{n,\ell}\]
(recall that $n\geq 6$ and $3\leq\ell\leq n-3$).  Then $\kappa^{-1}(p)=(\alpha,\ell+2)$, where
\[\alpha=(\underbrace{1,\ldots,1}_{\ell+1 \mbox{ copies}},\ell,\ell+1,\ell+4,\ell+5,\dots,n).\]
Thus,$(\alpha,\ell+2)\in\mathcal{P}_{n,\ell}^I$. However,
setting $i=\ell+2$ and $j=\ell+1$, we have
\[a_{i+1}=\ell+1\leq j\leq i, \quad\text{but}\quad
\disp_\alpha(C_j)=(\ell+1)-a_{\ell+1}=(\ell+1)-1=\ell.\]
So condition \textbf{(M2)} fails and $(\alpha,\ell+2)\notin\mathcal{P}_{n,\ell}^M$, proving strict containment as desired.
\end{proof}

We next consider the case $\ell=1$.

\begin{proposition} \label{Foata:ell-is-one}
For all $n\geq 1$, the Foata transform restricts to a bijection $\UPF_n\to\UPF_n$.
\end{proposition}

\begin{proof}
We begin by analyzing what generally happens to a UPF if we cycle a segment. Let $q\in\UPF_n$, and let $r$ be the parking function obtained by cycling a single interval $[a,b]$:
\[q = \left(\begin{array}{c|ccc|c}
\cdots&C_a &\cdots& C_b&\cdots\\
\cdots&q_a &\cdots& q_b&\cdots
\end{array}\right)\
\quad\to\quad
r = \left(\begin{array}{c|cccc|c}
\cdots&C_b & C_a &\cdots& C_{b-1}&\cdots\\
\cdots&q_b & q_a &\cdots& q_{b-1}&\cdots
\end{array}\right).\]
Here the bars are included for clarity.
Note that $\spot_r(C_b)\leq\spot_q(C_b)$ because passing from $q$ to $r$ moves $C_b$ toward the front of the queue. We show that $\disp_r(C_i)\leq 2$ for each $i\in [a,b].$

If $\spot_r(C_b)=\spot_q(C_b)$, then the Reserved Spot Lemma (\Cref{lemma:reserved-spot})
implies that all outcomes are preserved, so $r\in\UPF_n$.
Accordingly, we assume henceforth that $\spot_r(C_b)<\spot_q(C_b)$.  Since $q$ is a unit interval parking function, we have
\[0\leq \disp_r(C_b) = \spot_r(C_b)-q_b < \spot_q(C_b)-q_b = \disp_q(C_b) \leq 1\]
so the two inequalities are both actually equalities. In particular,
\begin{equation} \label{outb}
\spot_r(C_b)=q_b \quad\text{and}\quad \spot_q(C_b)=q_b+1.
\end{equation}
By interval rearrangement (\Cref{lemma:rearrange}),
\begin{align*}
\sum_{i=a}^b \spot_q(C_i) &= \sum_{i=a}^b \spot_r(C_i)
\intertext{so}
\sum_{i=a}^b \disp_q(C_i) &= \sum_{i=a}^b \disp_r(C_i)
\intertext{and}
1 + \sum_{i=a}^{b-1} \disp_q(C_i) &= \sum_{i=a}^{b-1} \disp_r(C_i). \end{align*}
Also, $\disp_r(C_i)\geq\disp_q(C_i)$ for all $i\in[a,b-1]$ since each $C_i$ now parks later. From these statements, we may conclude that there is a unique $j\in[a,b-1]$ such that
\begin{align} \label{dispj}
\disp_r(C_j)&=\disp_q(C_j)+1
\mbox{ and}\\
\label{dispi}
\disp_r(C_i)&=\disp_q(C_i) \quad \mbox{ for all }i\in[a,b-1]\setminus\{j\}.
\end{align}

If $\disp_r(C_j)= 1$ and $\disp_q(C_j)=0$, then  $\disp_r(C_i)\leq 2$ for each $i\in [a,b]$ and so we are done. Accordingly, assume for the sake of contradiction that $\disp_r(C_j)= 2$ and $\disp_q(C_j)=1$. 
Since $\disp_r(C_i)=\disp_q(C_i)$ for $i\in [a,b-1]\setminus\{j\},$ we find that $\spot_r(C_i)=\spot_q(C_i)$ for $i\in [a,b]\setminus\{j,b\}$. Recalling that $\{\spot_q(C_i)\suchthat i\in[a,b]\}=
\{\spot_r(C_i)\suchthat i\in[a,b]\}$,
we conclude that
\[\{\spot_r(C_b),\spot_r(C_j)\}=\{\spot_q(C_b),\spot_q(C_j)\},\]
which forces
\begin{align*}
\spot_q(C_j) &= \spot_r(C_b) = q_b, \mbox{ and}\\
\spot_r(C_j) &= \spot_q(C_b) = q_b+1.
\end{align*}
Combining with~\eqref{dispj}, we have
\[\disp_q(C_j) = 1 = \spot_q(C_j)-q_j = q_b-q_j\]
i.e., $q_j=q_b-1$.  
So it must be the case that spot $q_j$ was occupied by some car in a previous segment, and with respect to~$q$, cars $C_j$ and $C_b$ park in spots $q_j+1$ and $q_j+2$ respectively.  On the other hand, with respect to~$r$, car $C_b$ parks first in spot $q_j+1$, and car $C_j$ must park in spot $q_j+2$, giving it displacement 2,which is a contradiction.

 Now, let $p\in \UPF_n$. By induction on $n$, it is sufficient to show if that if $q=F_{n-1}(p)$ if a UPF, $r=F(p)$ is an UPF. Suppose separators are placed after entries $q_{j_1},\dots,q_{j_k}$ (and at the beginning) when computing $r$ from $q$. Note that 
\[\maxdisp(r)\leq 1 \Leftrightarrow \text{for each $i$, $\max\{\disp_r(C_k):k\in [j_{i-1}+1,j_i]\}\leq 1$}\]
where $j_0=0.$ By the first part of the proof, $\disp_r(C_k)\leq 2$ for all $k\in [j_{i-1}+1,j_i]$ for each $i$. Hence, all that remains to be show is that car $C_n$ has displacement at most $1$. However by Interval Rearrangement (\Cref{lemma:rearrange}), we have that $\spot_p(C_n)=\spot_q(C_n)$ and so $\disp_p(C_n)=\disp_q(C_n)\leq 1$ since $q$ is a UPF, as desired.
\end{proof}

Finally, we consider the most difficult case, when $\ell=2$.

\begin{proposition}\label{Foata:ell-is-two}
For all $n\geq 1$, if $p\in \IPF_n(2)$, then $F(p)\in \IPF_n(2)$.
\end{proposition}

\begin{proof}
  We will show that for every $m\leq n$, the partial Foata transform $F_m$ takes the set of 2-interval parking functions to the set of 2-interval parking functions.  This is trivial for $m\leq 2$ because $F_1$ and $F_2$ are the identity maps. For $m>2$, assume by way of induction that $p=(p_1,\dots, p_n)$ is a 2-interval parking function and suppose that $q=F_{m-1}(p)$ is a 2-interval parking function.

To keep track of individual cars, we assign them the license plates $C_1,\dots,C_n$ according to the order they park with respect to $q$.  Let $r=F_m(p)$.  To calculate $r$ from $q$ as in Defn.~\ref{defn:foata}, we first insert some set of $k\geq1$ separators, say after positions $i_1,\dots,i_k=m-1$, and then cycle the cars in every segment except the last one.  
\[q = \left(\begin{array}{ccc|ccc|c|ccc|cccc}
C_1 &\cdots& C_{i_1} & C_{i_1+1} &\cdots& C_{i_2} & \cdots & C_{i_{k-1}+1} &\cdots& C_{i_k} & C_m & C_{m+1} & \cdots & C_n\\
q_1 &\cdots& q_{i_1} & q_{i_1+1} &\cdots& q_{i_2} & \cdots & q_{i_{k-1}+1} &\cdots& q_{i_k} & p_m & p_{m+1} & \cdots & p_n
\end{array}\right)\]
and
\[r = \left(\begin{array}{ccc|ccc|c|ccc|cccc}
C_{i_1}\ C_1&\cdots& C_{i_1-1} & C_{i_2}\ C_{i_1+1} &\cdots& C_{i_2-1} & \cdots & C_{i_k}\ C_{i_{k-1}+1} &\cdots& C_{i_k-1} & C_m & C_{m+1} & \cdots & C_n\\
q_{i_1}\ q_1&\cdots& q_{i_1-1} & q_{i_2}\ q_{i_1+1} &\cdots& q_{i_2-1} & \cdots & q_{i_k}\ q_{i_{k-1}+1} &\cdots& q_{i_k-1} &  p_m & p_{m+1} & \cdots & p_n
\end{array}\right).\]
Now suppose towards a contradiction that there is some $j$ such that $\disp_r(C_j)>2$. By interval rearrangement (\Cref{lemma:rearrange}) $\disp_r(C_i)=\disp_q(C_i)\leq 2$ for all $i\geq m$, so $j<m$.  Therefore, there is some $h$ such that $i_{h-1}<j\leq i_h$. 
Again, interval rearrangement (\Cref{lemma:rearrange}) implies that we may focus on the segment containing $C_j$, without worrying about the cars before or after.

Car $C_{i_h}$ enters the parking lot earlier in $r$ than it did in $q$, so $\disp_r(C_{i_h})\leq \disp_q(C_{i_h})\leq 2$, which implies that $j\neq i_h$.

By definition of $F_m$, either (i) $q_{s}\leq p_m< q_{i_h}$ for all $i_{h-1}<s < i_h$, or (ii) $q_{s}> p_m\geq q_{i_h}$ for all $i_{h-1}<s< i_h$. If (ii) holds, then repeatedly applying \Cref{prop:sorting} implies that
\[\max \{\disp_r(C_s):i_{h-1}<s< i_h\}=\max \{\disp_q(C_s):i_{h-1}<s< i_h\}\leq 2\]
since $q$ is a 2-interval parking function. But this a contradiction as we assumed that $\disp_r(C_j)>2$. Now suppose (i) holds.

Suppose that in $q$, a car $C_s$ with $i_{h-1}<s< i_h$ parks in spot $q_{i_h}$. That can only happen if each of the spots $q_{s},q_{s+1},\dots,q_{i_h}-1$ is full. That includes $p_m$ since $q_s\leq p_m <q_{i_h}$ by assumption. Thus, car $C_m$ cannot park in its preferred spot $p_m$, it cannot park in spot $q_{i_h} $ (or any spot between $p_m$ and $q_{i_h}$), and it also cannot park in the spot $q_{i_h}+1$ because car $C_{i_h}$ will take that spot first if it is open. This forces $\disp_q(C_m)>2$, contradicting the inductive hypothesis.  We therefore conclude that there is no car $C_s$ with $i_{h-1}<s< i_h$ that parks in spot $q_h$. But then moving car $C_{i_h}$ to the front of its segment has no impact on the displacement of any car in that segment. In particular, $\disp_r(C_{j})\leq \disp_q(C_{j})\leq 2$, which is a contradiction since we assumed that $\disp_r(C_{j})>2$. Thus, we can conclude that $r=F_m(p)$ is a 2-interval parking function. We therefore conclude by induction that $F_m(p)\in \IPF_2(n)$ for any $p\in \IPF_2(n)$ and $m\leq n$.
\end{proof}

In the remainder of the paper, 
we return to the usual one-line notation for parking functions.

\section{Generating functions for displacement and inversions} \label{sec:counting_through_perms}

In this section, we study the generating function
\[\Phi_{n,\ell}=\Phi_{n,\ell}(q,t) = \sum_{\alpha\in \IPF_n(\ell)}q^{\disp(\alpha)}t^{\inv(\alpha)}.\]

The Foata transform $F$ preserves the content of a parking function and therefore its displacement. Therefore, \Cref{thm:main-Foata} has the following immediate consequence.

\begin{corollary}\label{cor:foata-disp-inv}
Let $n\geq 1$ and $\ell\in \{0,1,2,n-2,n-1\}$. Then
\[\Phi_{n,\ell}=\sum_{\alpha\in \IPF_n(\ell)}q^{\disp(\alpha)}t^{\inv(\alpha)}=\sum_{\alpha\in \IPF_n(\ell)}q^{\disp(\alpha)}t^{\maj(\alpha)}.\]
\end{corollary}

For the $\ell=0$ case (i.e. permutations), $\Phi_{n,\ell}(q,t)$ has the elegant and well known formula 
\[\Phi_{n,0}(q,t) = \prod_{i=1}^n [i]_q\]
where $[i]_q=1+q+\dots+q^{m-1}$. The polynomial $\Phi_{n,0}(q,t) $ is denoted $[n]_q!$ and referred to as the ``$q$-factorial''.

In \cite{inversions_in_pfs}, a subset of the authors gave closed formulas for the ``$q$-exponential generating functions'' $\sum_{n\geq 0} \Phi_{n,n-1}(1,q)\frac{z^n}{[n]_q!}$ and $\sum_{n\geq0} \Phi_{n,1}(1,q)\frac{z^n}{[n]_q!}$ tracking inversions among all parking functions and among unit interval parking functions respectively. 

We build on the approach taken by \cite{bib:ColmenarejoHarris} in studying the specialization of $\Phi_{n,n-1}$ which only tracks displacement. There they gave a formula for $\Phi_{n,n-1}(q,1)$ as a sum over permutations. Using their approach, we prove a formula for $\Phi_{n,\ell}(q,1)$ in \cref{thm:disp-enumerator-for-ell-interval} by partitioning $\IPF_n(\ell)$ into the fibers of the function $\Out_{n,\ell}\colon\IPF_n(\ell)\to\Sym_n$
defined by $\Out_{n,\ell}(\alpha)=\car_\alpha$. Following \cite{bib:ColmenarejoHarris}, this formula can be upgraded to a $(q,t)$-version tracking parking function statistics which are derived from permutation statistics. In \cref{thm:disp-inv-enumerator-for-upfs} we use the specialization of \cref{thm:disp-enumerator-for-ell-interval}  to obtain a particularly nice formula for $\Phi_{n,1}(q,t)$. We explore cyclic sieving phenomena for $\Phi_{n,1}(q,t)$ in \cref{section:csp}, and in \cref{sec:1to2} we apply another ``counting by fibers'' technique to obtain a formula for $\Phi_{n,2}(q,t)$ as a sum over unit interval parking functions.

For the case $\ell=n-2$, we can write $\Phi_{n,n-2}$ in terms of displacement-inversion enumerators for (classical) parking functions.  This also follows from the description of $\PF_n\setminus\IPF_n(n-2)$ in \cref{ex:basic counts}. 

\begin{corollary} \label{ell-minus-two}
For all $n\geq 2$, 
\begin{align*}
    \Phi_{n,n-2}(q,t)
    &=\Phi_{n,n-1}-(qt)^{n-1}\sum_{\beta\in \PF_{n-1}}q^{\disp(\beta)}t^{\inv(\beta)-\ones(\beta)}\\
    &=\Phi_{n,n-1}-(qt)^{n-1}\sum_{\beta\in \PF_{n-1}}q^{\disp(\beta)}t^{\maj(\beta)-\ones(\beta)}\\
    &=\Phi_{n,n-1}-(qt)^{n-1}\sum_{\beta\in \PF_{n-1}}q^{\disp(\beta)}t^{\maj(\beta)-\mathrm{lastisone}(\beta)}
\end{align*}
where $\ones(\beta)=|\{i\in[n-1]\suchthat b_i=1\}|$ and $\mathrm{lastisone}(\beta)$ is defined to be $n-1$ if $\beta_{n-1}=1$ and $0$ otherwise.
\end{corollary}
\begin{proof}
    Recall that $\PF_n\setminus\IPF_n(n-2)$ is the set of parking functions $\beta1$ with $\beta\in\PF_{n-1}$. Such a parking function inherits the inversions of the $\beta$ and has an additional inversion for each preference of $\beta$ that is not equal to 1. This proves the first equation. The second equation follows by applying the Foata transform. The final equation is by noting that $\maj(\beta1)=\maj(\beta)+(n-1)-\mathrm{lastisone}(\beta)$. 
\end{proof}

\subsection{Enumeration by counting through permutations}
For a permutation $\sigma=\sigma_1\sigma_2\cdots\sigma_n\in\mathfrak{S}_n$ and position $i\in[n]$, define
\begin{equation} \label{define-Ll}
L_\ell(i;\sigma)=\min(\ell+1,i-t+1),
\end{equation}
where $\sigma_t, \sigma_{t+1}, \ldots, \sigma_{i}$ is the longest contiguous subsequence of $\sigma$ such that  $\sigma_k \leq \sigma_i$ for all $k\in[t,i]$.  
That is, $L_\ell(i;\sigma)$ is the number of possible preferred spots for car $\sigma_i$ that results in car $\sigma_i$ parking in spot~$i$, subject to the choices of the previous cars and the constraint that the result must be an $\ell$-interval parking function.  It follows that for all $\ell\in\N$ and  $\sigma\in\mathfrak{S}_n$, we have 
\begin{equation} \label{fibersizes}
|\Out_{n,\ell}^{-1}(\sigma)|=\prod_{i=1}^nL_\ell(i;\sigma).
\end{equation}

The $\ell=n-1$ case of~\eqref{fibersizes} is \cite[Eqns.~2.1, 2.2]{Konheim1966}; see also \cite[Proposition 3.1]{bib:ColmenarejoHarris} and \cite[Exercise 5.49d]{stanley1999enumerative}.  Our $L_{n-1}(i;\sigma)$ corresponds to $\tau_{i,n}(\sigma)$ in \cite{Konheim1966} and to
$\ell(i;\sigma)$ in \cite{bib:ColmenarejoHarris}.

In the case that $\sigma=\id$ is the identity permutation, the number $L_\ell(i,\id)$ is either $i$ (if $i\leq\ell$) or $\ell+1$ (if $\ell+1\leq i\leq n$), so $|\Out_{n,\ell}^{-1}(\id)|=(\ell+1)^{n-\ell}\cdot\ell!$.
These numbers occur as sequence \#\seqnum{A299504} in \cite{OEIS}, where they count certain permutations; here, they count certain parking functions.

It follows from~\eqref{fibersizes} that 
\begin{equation}\label{New formula for ell interval pfs}
|\IPF_n(\ell)|=\sum_{\sigma\in\mathfrak{S}_n} \prod_{i=1}^n L_\ell(i;\sigma).
\end{equation}
This is more efficient for explicit computation than iterating over all parking functions and picking out those with maximum displacement at most $\ell$.  In general, we do not have a closed formula for $|\IPF_n(\ell)|$ except in extreme cases; see \cref{ex:basic counts}.

In \cite[Props.~5.2, 5.3]{bib:ColmenarejoHarris}, the authors provide $q$-analogues of equations~\eqref{fibersizes} and~\eqref{New formula for ell interval pfs} for parking functions (as opposed to $\ell$-interval parking functions) using total displacement. They extend this to a $(q,t)$-analogue in \cite[Cor.~5.3]{bib:ColmenarejoHarris} using permutation statistics. We can obtain a similar $(q,t)$-analogue for $\ell$-interval parking functions.

\begin{theorem}\label{thm:disp-enumerator-for-ell-interval}
Let $0\leq\ell\leq n-1$. For all $\sigma\in \S_n$, we have 
\begin{equation} \label{fibersizesq}
\sum_{\alpha\in \Out_{n,\ell}^{-1}(\sigma)}q^{\disp(\alpha)}=\prod_{i=1}^n[L_\ell(i;\sigma)]_q,
\end{equation}
where $[m]_q=1+q+\cdots+q^{m-1}$.
Therefore,
\begin{equation} \label{dispalphainvcar}
 \sum_{\alpha\in \IPF_n(\ell)}q^{\disp(\alpha)}t^{\mathrm{stat}(\car_\alpha)}=\sum_{\sigma\in \S_n}t^{\mathrm{stat}(\sigma)}\prod_{i=1}^n[L_\ell(i;\sigma)]_q.
\end{equation}
for any permutation statistic $\mathrm{stat}:\S_n\to \mathbb{C}$.
\end{theorem}
\begin{proof}
   Let $\sigma\in\S_n$ and $\alpha=(a_1,a_2,\dots,a_n)\in\Out_{n,\ell}^{-1}(\sigma)$. For each $i$, if car $\sigma_i$ parks in spot $i$, then $a_i$
   must be selected from $[i-L_\ell(i;\sigma)+1,i]$. Hence, for each $j\in \{0,\dots,L_\ell(i;\sigma)-1\}$, the selection of $a_i=i-j\in [i-L_\ell(i;\sigma)+1,i]$ contributes $j$ to the total displacement. Since $[L_\ell(i;\sigma)]_q=1+q+\cdots+q^{L_\ell(i;\sigma)-1}$, we obtain~\eqref{fibersizesq}. For~\eqref{dispalphainvcar}, observe that
   \[\sum_{\alpha\in\IPF_n(\ell)}q^{\disp(\alpha)}t^{\mathrm{stat}(\car_\alpha)}=\sum_{\sigma \in \Sym_n} \lp\sum_{\alpha\in \Out_{n,\ell}^{-1}(\sigma)}q^{\disp(\alpha)}\rp  t^{\mathrm{stat}(\sigma)}\]
   and then apply~\eqref{fibersizesq}. 
\end{proof}

The $\ell=n-1$ case is \cite[Cor.~5.3]{bib:ColmenarejoHarris}. Also note that nothing stops us from including more variables and tracking multiple statistics at once with essentially the same formula.

In the case $\ell=1$ (unit interval parking functions), the formula of \cref{thm:disp-enumerator-for-ell-interval} can be written in a very nice formula in terms of classical permutation statistics. Recall from \cref{lem: inv of outcome} that $\Inv(\alpha)=\Inv(\spot_\alpha)$ for $\alpha\in \UPF_n$. This means that for unit interval parking functions, \eqref{dispalphainvcar} can be used to tell us about inversions, descents, and major index as computed on the parking functions (and not just their associated permutations). 

\begin{theorem}\label{thm:disp-inv-enumerator-for-upfs}
For all $n\geq 1$,
\begin{align}
    \Phi_{n,1}(q,t)&=\sum_{\sigma\in \S_n}(1+q)^{\asc(\sigma)}t^{\inv(\sigma)} \label{eq:upfphi1}\\
    &=\sum_{\sigma\in \S_n}(1+q)^{\asc(\sigma)}t^{\maj(\sigma^{-1})} \label{eq:upfphi2}
\end{align}
Moreover,
\begin{equation}\label{eq:upfphi3}
    \sum_{\alpha\in \UPF_n} q^{\disp(\alpha)}t^{\inv(\alpha)}x^{\des(\alpha)}y^{\maj(\alpha)}
    =\sum_{\sigma\in \S_n}(1+q)^{\asc(\sigma)}t^{\inv(\sigma)}x^{\des(\sigma^{-1})}y^{\maj(\sigma^{-1})}.
    \end{equation}
\end{theorem}
\begin{proof}
Taking $\ell=1$ in \eqref{define-Ll}, we see that
\begin{equation*} 
L_1(i;\sigma)=\begin{cases}
2&\text{if $i-1$ is an ascent of $\sigma$,}\\
1&\text{otherwise.}
\end{cases}
\end{equation*}
Since $\Inv(\alpha)=\Inv(\spot_\alpha)$, we get that $\inv(\alpha)=\inv(\spot_\alpha)=\inv(\car_\alpha^{-1})=\inv(\car_\alpha)$. Additionally,  $\des(\alpha)=\des(\car_\alpha^{-1})$ and $\maj(\alpha)=\maj(\car_\alpha^{-1})$.
Now \Cref{eq:upfphi1} and \Cref{eq:upfphi2} follow directly from setting $\ell=1$ in \Cref{dispalphainvcar} with $\mathrm{stat}(\sigma)=\inv(\sigma)$ and $\mathrm{stat}(\sigma)=\maj(\sigma^{-1})$ respectively. \Cref{eq:upfphi3} follows from the logic of the proof of \cref{thm:disp-enumerator-for-ell-interval}. 
\end{proof}

Setting $q=t=1$ in \Cref{thm:disp-inv-enumerator-for-upfs} recovers~\eqref{eq:ascent-power-of-2-are-fubini}.
More generally, setting $t=1$ produces
\begin{equation}\label{eq:stirling-to-ascents}
\Phi_{n,1}(q,1)=\sum_{\alpha\in \UPF_n}q^{\disp(\alpha)}
=\sum_{k=1}^{n} k!S(n,k)q^{n-k}
=\sum_{\sigma\in \S_n}(1+q)^{\asc(\sigma)},
\end{equation}
where $S(n,k)$ denotes the Stirling number of the second kind; the second equality is \cite[p.~269, eqn.~(6.39)]{GKP}.  Note that $k!S(n,k)$ is the number of unit interval parking functions with displacement $n-k$.  
Extracting the $q^1$ coefficient from the second and fourth expressions in~\eqref{eq:stirling-to-ascents}, we recover the enumerative result \cite[Thm.~4]{Hanoi} that $|\{\alpha\in \UPF_n\suchthat \disp(\alpha)=1\}|$ is the $n$th Lah number $n!(n-1)/2$ (\#\seqnum{A001286} of \cite{OEIS}). See \Cref{cor: CSP-evaluation} for evaluations of $\Phi_{n,1}(q,t)$ at other $t$.

\begin{remark}
Setting $q=-1$ in \Cref{thm:disp-inv-enumerator-for-upfs} yields
    \begin{equation}\label{eq:minus-1-eval-disp-upf}
    \Phi_{n,1}(-1,t)=\sum_{\alpha\in \UPF_n}(-1)^{\disp(\alpha)}t^{\inv(\alpha)}
    =t^{\binom{n}{2}}
    \end{equation}
since the permutation $w_0(i)=n-i+1$ is the only one with no ascents. 
This formula can also be proven combinatorially by constructing a sign-reversing, inversion-preserving involution on unit interval parking functions. Given $\alpha \in \UPF_n\setminus\{w_0\}$, let $\pi_1\mid\pi_2\mid\cdots\mid\pi_m$ be its block structure (see \Cref{thm:upf_rearrangement}), and let $k$ be the smallest index such that either (i) $|\pi_k|\geq 2$, or (ii) $|\pi_k|=1$ and $\pi_k$ appears earlier in $\alpha$ than every member of $\pi_{k+1}$.
(The condition $\alpha\neq w_0$ implies that at least one of conditions (i), (ii) must be satisfied.) Then, define $\varphi(\alpha)$ by modifying $\alpha$ as follows:
\begin{enumerate}
    \item If $\pi_k=(j,j,j+1,j+2,\dots)$, then replace the second occurrence of $j$ by $j+1$.
    \item If $\pi_k=(j)$, then replace the first occurrence of $j+1$ in $\alpha$ by $j$.
\end{enumerate}
The map $\varphi$ is a fixed-point-free involution on $\UPF_n\setminus\{w_0\}$ satisfying $\disp(\varphi(\alpha))=\disp(\alpha)\pm1$ and $\inv(\varphi(\alpha))=\inv(\alpha)$ (we omit the proof).  If one associates unit interval parking functions with faces of the permutohedron, as done in~\cite{unit_perm}, the function $\varphi$ appears to give an acyclic matching in the sense of discrete Morse theory~\cite{DMT2,DMT1}.
\end{remark}

\subsection{A cyclic sieving phenomenon}\label{section:csp}

Let $X$ be a finite set, let $f(t)$ be a polynomial in~$t$ such that $f(1)=|X|$, let $C=\langle g\rangle$ be a cyclic group of order~$n$ acting on $X$, and let $\omega$ be a primitive $n$th root of unity.  The triple $(X,f(t),C)$ is said to exhibit the \defterm{cyclic sieving phenomenon} \cite{reiner2004cyclic} if 
$f(\omega^j)=|\{x\in X\suchthat g^j\cdot x=x\}|$ for any $j\in [n]$.

 For a given surjective function $s:[n]\to [m]$ and permutation $\sigma\in S_n$, define 
\[\sigma\cdot s := s\circ \sigma^{-1}.\]
It is easy to check this defines an action of $\S_n$ on the set of surjective functions $s:[n]\to [m]$. Using the bijection in  Corollary~\ref{cor:block-bijection}, we extend this to an action on $\UPF_n$. For example, if $\alpha=(8,1,5,5,1,2,4,7)\in \UPF_8$ and $\sigma=(1,8)(2,7,5)(3,4)(6)$, then we compute $\sigma\cdot \alpha$ by first turning $\alpha$ into the surjective function $s=51331124$, computing $\sigma\cdot s=41332115$, and then converting back to $(7,1,5,5,4,2,1,8).$ We can further see that two UPFs are in the same orbit of this $\S_n$-action if and only if they have the same block structure. This action is considered implicitly in \cite{unit_perm} as an $\S_n$-action on the faces of the permutohedron.

We restrict the above $\S_n$-action on $\UPF_n$ to the cyclic $C_n$ generated by the $n$-cycle $g:=(1,2,\dots,n)$. For cyclic sieving, we need to characterize the fixed points of the action.

\begin{lemma}\label{lem:fixed-pts-UPFn-action}
    Let $\alpha=(a_1,\dots,a_n)\in \UPF_n$, let $j$ be a divisor of $n$, and suppose $n=jd.$
    \begin{enumerate}
        \item $g^j \cdot \alpha=\alpha$ if and only if for each $i\in [n]$, the entries
    \[a_i,a_{i+j},\dots,a_{i+(d-1)j},\]
    where the indices are taken modulo $n$, all belong to the same block.
    \item If $n=jd$ and $\alpha$ has block structure $\pi_1|\pi_2|\cdots|\pi_m$, then there are 
    \[\binom{n/d}{|\pi_1|/d,|\pi_2|/d,\dots,|\pi_m|/d}\]
elements in the $\S_n$-orbit of $\alpha$ that are fixed by $g^j$, where we adopt the convention that any multinomial expression involving one or more non-integers is zero.
    \end{enumerate} 
\end{lemma}
\begin{proof}
    Let $s$ be the corresponding surjective function of $\alpha$. Since $g^j\cdot \alpha = \alpha$ if and only if $g^j \cdot s=s$, and $(g^j\cdot s)(i)=s(g^{-j}(i))=s(i+j),$ we have that $g^j\cdot \alpha=\alpha$ if and only if $s(i)=s(i+j)$ for all $i$. Claim (1) then follows from translating back to UPFs.

    
    
    For claim (2), if $g^j \cdot s = s$, then by part (1), the entire function is determined by its restriction $s|_{[j]}.$ This is also a surjective function $s|_{[j]}:[j]\to [m]$ and the corresponding UPF in $\UPF_{j}$ has blocks of sizes $|\pi_1|/d,\dots,|\pi_m|/{d}$. Thus, claim (2) follows from counting the number of such UPFs via \Cref{thm:upf_rearrangement}.
\end{proof}
Recall that for a composition $c=(c_1,\dots,c_\ell)$ of $n$, the $t$-multinomial coefficient is 
\[\qbinom{n}{c_1,c_2,\dots,c_\ell}_t:=\frac{[n]_t!}{[c_1]_t![c_2]_t!\cdots [c_\ell]_t!},\]
where for any positive integer $k$, we define $[k]_t=1+q+\cdots+q^{k-1}$ and $[k]_t!:=[k]_t[k-1]_t\cdots [2]_t[1]_t$. Also set $[0]_t=[0]_t!=1.$ A useful fact for cyclic sieving is the following: if
$\omega$ is a primitive  $n$th root of unity and $d$ is a divisor of $n$, then
\begin{equation}\label{eq:qbinom-at-root-of-unity-general}
\qbinom{n}{c_1,c_2,\dots,c_\ell}_{t=\omega^{n/d}}
=\binom{\frac{n}{d}}{\frac{c_1}{d},\frac{c_2}{d},\dots,\frac{c_k}{d}}.
\end{equation}
 See \cite[Equation (4.5)]{reiner2004cyclic}.



For $k\in \N$, let $\UPF_{n,k}^{\disp}$ be the set of unit interval parking functions with total displacement $k$
and set 
\[f_{n,k}^{\disp}(t)=\sum_{\alpha\in\UPF_{n,k}^{\disp}}t^{\inv(\alpha)},\]
so that $\Phi_{n,1}(q,t)=\sum_k q^k f_{n,k}^{\disp}(t)$.

\begin{theorem} \label{thm:cyclic-sieving}
The triple $(\UPF_{n,k}^{\disp},f_{n,k}^{\disp}(t),C_n)$ satisfies the cyclic sieving phenomenon for each $n$ and $k$. 
That is, for all $n,k,j$, if $\omega$ is a primitive $n$th root of unity, then,
\begin{align}\label{righthandsideThm3.6}f_{n,k}^{\disp}(\omega^j) &= \sum_{\alpha\in \UPF_{n,k}^{\disp}} (\omega^j)^{\inv(\alpha)}= |\{\alpha\in \UPF_{n,k}^{\disp}\suchthat  g^j\cdot \alpha= \alpha\}|.
\end{align}
\end{theorem}

\begin{proof}
Let $\omega$ be a primitive $n$th root of unity.  
The conjugacy class of $g^j$ in $\S_n$, and therefore the third expression in~\Cref{righthandsideThm3.6}, depends only on $\gcd(j,n)$, so it suffices to consider the case that $j$ divides $n$, so that $\omega^j$ is a primitive $d$th root of unity, where $d=n/j$.
For each $\beta\in\UPF_n^\uparrow$, let $\S_n\cdot \beta$ be its $\Sym_n$-orbit, and write its block structure as $\pi^\beta_1\mid \pi^\beta_2\mid \cdots\mid \pi^ \beta_{m_\beta}$.

Recall that the elements of each block of $\alpha\in\UPF_n$ occur in weakly increasing order in $\alpha$, left to right.  Therefore, the number of inversions of~$\alpha$ is unchanged by replacing each element in its $j$th block with the number~$j$.  Accordingly, 
\begin{align}\nonumber
\sum_{\alpha\in \mathfrak{S}_n}t^{\inv(\alpha)}&=\sum_{u\in \S_n \cdot w_\beta}t^{\inv(u)}\\
&=\qbinom{n}{|\pi^\beta_1|,|\pi^\beta_2|,\;\;\dots,\;|\pi^\beta_{m_\beta}|}_t,\label{eq:to ref}
 \end{align}
where $w_\beta={1\cdots 1}2\cdots 2\cdots {m_\beta\cdots m_\beta}$ is a word with $|\pi^\beta_1|$ 1's, $|\pi^\beta_2|$ 2's, etc, written in weakly increasing order. 
The equality in \eqref{eq:to ref} is \cite[Prop.~1.7.1]{stanley2012enumerative}, originally due to MacMahon \cite[pp.~314-315]{MacMahon}.
Now
\begin{align}
\sum_{\alpha\in \UPF_n}q^{\disp(\alpha)}(\omega^j)^{\inv(\alpha)}
&=\sum_{\beta\in\UPF^\uparrow_n} q^{\disp(\beta)}\sum_{\alpha\in \S_n\cdot \beta} (\omega^j)^{\inv(\alpha)}\notag\\
&=\sum_{\beta\in\UPF^\uparrow_n} q^{\disp(\beta)} \qbinom{n}{|\pi^\beta_1|,|\pi^\beta_2|,\;\dots,\;|\pi^\beta_{m_\beta}|}_{t=\omega^j}\notag\\
&= \sum_{\beta\in\UPF^\uparrow_n} q^{\disp(\beta)} \binom{n/d}{|\pi^\beta_1|/d,\;|\pi^\beta_2|/d,\;\dots,\;|\pi^\beta_{m_\beta}|/d},\label{sieve:1}\\
&= \sum_{\beta\in\UPF^\uparrow_n} q^{\disp(\beta)} |\{\alpha\in \S_n\cdot \beta\suchthat g^j\cdot \alpha = \alpha\}|\label{sieve:2}
\end{align}
where the equality in \eqref{sieve:1} follows from~\eqref{eq:qbinom-at-root-of-unity-general} and the equality in \eqref{sieve:2} follows from \Cref{lem:fixed-pts-UPFn-action}.
Thus, we have shown that
\begin{equation} \label{sieve:2}
\sum_{\alpha\in \UPF_n}q^{\disp(\alpha)}(\omega^j)^{\inv(\alpha)}
= \sum_{\beta\in\UPF^\uparrow_n} q^{\disp(\beta)} |\{\alpha\in \S_n\cdot \beta \suchthat g^j\cdot \alpha = \alpha\}|
\end{equation}
and grouping together unit interval parking functions with the same total displacement, we have
\[\sum_{\alpha\in \UPF_{n,k}^{\disp}}
(\omega^j)^{\inv(\alpha)}= |\{\alpha\in \UPF_{n,k}^{\disp}\suchthat  g^j\cdot \alpha= \alpha\}|,\]
as desired.
\end{proof}
As a consequence, we get some nice combinatorial interpretations of $\Phi_{n,1}(q,t)$ when we evaluate $t$ at roots of unity.

\begin{corollary}\label{cor: CSP-evaluation}
Let $n$ be a positive integer.
\begin{enumerate}
\item  If $\omega$ is a primitive $n$th root of unity, then
    \[\Phi_{n,1}(q,\omega)=\sum_{\alpha\in \UPF_n}q^{\disp(\alpha)}\omega^{\inv(\alpha)}=q^{n-1}.\]
\item If $n$ is even, then
   \[\Phi_{n,1}(q,-1)=\sum_{\alpha\in \UPF_n}q^{\disp(\alpha)}(-1)^{\inv(\alpha)}
   = q^{n/2}\sum_{\alpha\in \UPF_{n/2}}q^{\disp(\alpha)}.\]
\end{enumerate}   
\end{corollary}
\begin{proof}
(1) This is the case $j=1$ of \Cref{thm:cyclic-sieving}.  A UPF is fixed by $g$ if and only if it has a single block of size $n$.  There is exactly one such UPF, namely $\alpha=(1,1,2,\dots,n-1)$, which has total displacement $n-1$.

(2) Define $\UPF^{2}_n$ to be the set of unit interval parking functions whose block sizes are all divisible by 2 and let $\UPF_n^{2\uparrow}=\UPF_n^2\cap \UPF_n^{\uparrow}$.  
Since an increasing UPF is determined by its list of block sizes, in order, there is a bijection $\rho\colon\UPF_n^{2\uparrow} \to \UPF_{n/2}^{\uparrow}$ given by dividing all block sizes by~2. (For instance, $\rho(1123\mid55\mid7789)=11\mid3\mid44$, because these are the unique increasing UPFs with block sizes $4,2,4$ and $2,1,2$ respectively.)  Note also that $\disp(\beta)=\disp(\rho(\beta))+n/2$.  Now substituting $j=n/2$ in~\eqref{sieve:1} yields
\begin{align*}
\sum_{\alpha\in \UPF_n}q^{\disp(\alpha)}(-1)^{\inv(\alpha)}
&= \sum_{\beta\in\UPF_n^{2\uparrow}} q^{\disp(\beta)} \binom{n/2}{|\pi^\beta_1|/2,\pi^\beta_2|/2,\;\dots,\;|\pi^\beta_{m_\beta}|/2} \\
&= q^{n/2}\sum_{\alpha\in \UPF_{n/2}^{\uparrow}}q^{\disp(\alpha)}\binom{n/2}{|\pi^\alpha_1|,|\pi^\alpha_2|,\;\dots,\;|\pi^\alpha_{m_\alpha}|}
  && \text{(setting $\alpha=\phi(\beta)$)}\\
&=q^{n/2}\sum_{\alpha\in \UPF_{n/2}}q^{\disp(\alpha)},
\end{align*}
where the last step follows from \eqref{sieve:1}, replacing $n$ and $j$ with $n/2$ and $d$ with 1.
\end{proof}

\subsection{From \texorpdfstring{$\ell=1$}{l=1} to \texorpdfstring{$\ell=2$}{l=2}}\label{sec:1to2}

\Cref{thm:disp-inv-enumerator-for-upfs} can be viewed as reducing the enumeration of $1$-interval parking functions to the enumeration of $0$-interval parking functions, which are just permutations. 
We now describe a similar reduction from $2$-interval parking functions to $1$-interval parking functions, although it is more complicated.

Let $\beta=(b_1,b_2,\ldots,b_n)\in \UPF_n$ have block structure $\pi_1\mid \pi_2\mid \cdots \mid \pi_m$.
Let $\RRR(\beta)$ be the set of numbers $i\in[n]$ such that $ b_{i}$ is the second entry in its block $\pi_j$, where $j>1$,  and $b_i$ occurs after the last entry of block $\pi_{j-1}$
Let $\SSS(\beta)$ be the set of numbers $i\in[n]$ such that $ b_{i}$ is the third or later entry in its block $\pi_j$.  Then
\begin{equation} \label{S-formula}
|\SSS(\beta)|=\sum_{i=1}^m\max(|\pi_i|-2,0).
\end{equation}

\begin{example}\label{ex:making sense of defs}
If $\beta=7511278935\in\UPF_{10}$, then its block structure is $\pi_1\mid\pi_2\mid\pi_3=1123\mid55\mid7789$. The second entry in block $\pi_2$ is $b_{10}=5$ and appears after the last entry of $\pi_1$, which is $b_9=3$. Hence,
$10\in\RRR(\beta)$. 
The second entry in block $\pi_3$ is $b_6=7$ and does not appear after the last entry of $\pi_2$.
Hence, we can see that $6\notin\RRR(\beta)$. 
Thus, we have $\RRR(\beta)=\{10\}$. 
Lastly, in block $\pi_1=1123$ the third or later entries are $b_5=2$ and $b_9=3$, while in block $\pi_3=7789$ the third or later entries are $b_7=8$ and $b_8=9$. 
We then get that $\SSS(7511278935)=\{5,7,8,9\}$.
We have $|\SSS(\beta)|=2+0+2=4$, which as expected by \eqref{S-formula}.
\end{example}

Observe that
\begin{align}
\RRR(\beta)\cap\SSS(\beta)&=\emptyset, \ \text{and}\label{S-intersect-S}\\
\RRR(\beta)\cup\SSS(\beta)&=\{i\in[2,n]\suchthat \text{ spots $ b_i-1$ and $ b_i$ are both occupied when car $i$ parks}\}.\label{R-union-S}
\end{align}
In particular, if $i\in R(\beta)$, then spot $ b_i-1$ is occupied by the last car of block $\pi_{j-1}$, and spot $ b_i$ by the first car of block $\pi_j$.

\begin{example}
Continuing \Cref{ex:making sense of defs}, let $\beta=7511278935$ and we can confirm $\RRR(\beta)\cap\SSS(\beta)=\emptyset$. 
One can readily confirm that the only indices $i\in[2,n]$ such that spots $b_i-1$ and $b_i$ are occupied when car $i$ parks are precisely the indices $i\in\{5,7,8,9,10\}=\RRR(\beta)\cup\SSS(\beta)$. For example, when car $7$ parks, spots $b_7-1=8-1=7$ and $b_7=8$ are occupied by cars $1$ and $6$, respectively.
\end{example}

Define a map $\eta:\IPF_n(2)\to\UPF_n$ by
\[\eta(\alpha)_i=\begin{cases}
a_i+1&\disp_\alpha(i)=2,\\
a_i&\disp_\alpha(i)\leq 1.\\
\end{cases}\]

Informally, we obtain $\eta(\alpha)$ by asking each car with displacement~2, ``Please lower your expectations by one spot.''  It is unsurprising, though not entirely obvious, that the result is a unit interval parking function with the same spot permutation, as we now prove.

\begin{lemma}\label{lem:phi2-well-def}
Let $\alpha\in\IPF_n(2)$ and $\beta=\eta(\alpha)$.  Then $\beta\in\UPF_n$, and $\spot_\beta=\spot_\alpha$.
\end{lemma}

\begin{proof}
We show by induction that $\spot_\beta(i)=\spot_\alpha(i)$ and $\disp_\beta(i)\leq1$ for all $i\in[n]$. 
For $i=1$, we have $\disp_\alpha(1)=0$, so $ b_1=a_1=\spot_\beta(1)=\spot_\beta(1)$.  
For the inductive step, suppose that $\spot_\alpha(j)=\spot_\beta(j)$ for all $j<i$.
\begin{itemize}
\item If $\disp_\alpha(i)\leq 1$, then $ b_i=a_i$ and $\spot_\beta(i)=\spot_\alpha(i)$; in particular, $\disp_\beta(i)\leq 1$.
\item If $\disp_\alpha(i)=2$, then $\spot_\alpha(i)=a_i+2$.  In particular, spots $a_i$ and $a_i+1$ are full when car $i$ parks.
Changing the preference of the $i$th car from $a_i$ to $ b_i=a_i+1$ does not affect where car $i$ parks, but then $\disp_\beta(i)=a_i+2- b_i=1$.\qedhere
\end{itemize}
\end{proof}

We now prove our main enumeration result on $2$-interval parking functions.

\begin{theorem}\label{thm:disp-enumerator-for-2-interval}
For all $n\geq 1$, 
\[\Phi_{n,2}(q,t)=\sum_{\beta\in \UPF_n} q^{\disp(\beta)}t^{\inv(\beta)}(1+q)^{|\SSS(\beta)|}(1+qt)^{|\RRR(\beta)|}.\]
In particular,
\[|\IPF_n(2)|=\sum_{\beta\in \UPF_n}2^{|\SSS(\beta)|+|\RRR(\beta)|}.\]
\end{theorem}

\begin{proof}
Let $\beta=(b_1,b_2,\ldots,b_n)\in \UPF_n$ have block structure $\pi_1\mid \pi_2\mid \cdots\mid\pi_m$.  Let $\RRR=\RRR(\beta)$ and $\SSS=\SSS(\beta)$.  For $R\subseteq\RRR$ and $S\subseteq\SSS$, define $\alpha=(a_1,a_2,\ldots,a_n)=g_\beta(R,S)$ by
\[a_i=\begin{cases}
     b_i-1&\text{ if }i\in R\cup S,\\
     b_i&\text{ if }i\not\in R\cup S.
\end{cases}\]
Then $\alpha$ is a parking function by the rearrangement criterion.
In light of the description of $\RRR\cup \SSS$ in~\eqref{R-union-S}, it follows that for all $i\in[n]$, 
\begin{subequations}
\begin{align}
i\in R\cup S & \mbox{ if and only if } \disp_\alpha(i)=2;\label{gclaim:1}\\
i\not\in R\cup S &\mbox{ if and only if }  \disp_\alpha(i)\leq 1;\label{gclaim:2}\\
\spot_\alpha(i)&=\spot_\beta(i).\label{gclaim:3}
\end{align}
\end{subequations}
 It follows that $\alpha\in\eta^{-1}(\beta)\subseteq\IPF_n(2)$.  This defines a function
\[g_\beta\colon2^\RRR\x2^\SSS\to\eta^{-1}(\beta).\]
We claim that $g_\beta$ is a bijection.  
It is injective since $\RRR\cap\SSS=\emptyset$.  To show that $g_\beta$ is surjective, it suffices to show that if $\alpha\in \eta^{-1}(\beta)$ and $\disp_\alpha(i)=2$, then $i\in \RRR\cup \SSS$.  Indeed, we have $b_i=a_i+1$ by definition of $\eta$.  Since $\disp_\alpha(i)=2$, when car $i$ parks, spots $a_i= b_i-1$ and $a_i+1= b_i$ are both occupied, and since $\spot_\alpha=\spot_\beta$ by \Cref{lem:phi2-well-def}, we have $i\in \RRR\cup \SSS$ by~\eqref{R-union-S}. 

Note that the inverse map $f_\beta=g_\beta^{-1}$ is given by $f_\beta(\alpha)=(R,S)$, where
  \[R=\RRR\cap \{i\in [n]\suchthat \disp_\alpha(i)=2\},~\mbox{and }
  S=\SSS\cap \{i\in [n]\suchthat \disp_\alpha(i)=2\}.\]
Now, let $\beta\in\UPF_n$, and let $\alpha=g_\beta(R,S)$ for some $R\subseteq\RRR(\beta)$ and $S\subseteq\SSS(\beta)$.  Then

\[\disp_\alpha(i)=\begin{cases} \disp_\beta(i) & \text{ if } i\notin R\cup S,\\ \disp_\beta(i)+1 & \text{ if } i\in R\cup S,\end{cases}\]
so
\begin{equation}\label{stats:disp}
\disp(\alpha)=\disp(\beta)+|R|+|S|.
\end{equation}
  
We now describe $\inv(\alpha)$ in terms of $\inv(\beta)$. Suppose for contradiction that $(i,j)$ is an inversion of $\beta$ but not of $\alpha$. Then, by the definition of $g_\beta$, we have $i\in R\cup S$, $j\not\in R\cup S$, and $ b_i= b_j+1$.  Hence, $a_i =  b_i - 1=   b_j = a_j$. 
Since $\beta$ is a unit interval parking function, it does not have inversions within blocks (\Cref{thm:upf_rearrangement}, part~(2)); hence, $ b_i$ and $ b_j$ must be in different blocks. Thus, $ b_i$ is either the first or second element of some block $\pi_{k+1}$ and $ b_j$ is the last element of  block $\pi_k$. If $ b_i$ is the first element of $\pi_{k+1}$, then $i\notin \RRR\cup \SSS.$  If $ b_i$ is the second element of $\pi_{k+1}$ then again $i\notin \RRR\cup \SSS$ because $ b_i$ occurs before the only entry of $\pi_k$, namely $ b_j.$ Hence, we get that $i\notin \RRR\cup \SSS$, which is a contradiction. We conclude that $\Inv(\beta)\subseteq \Inv(\alpha).$

Now suppose that $(i,j)$ is an inversion of $\alpha$ but not of $\beta$. 
Then we have by definition of $g_\beta$ that $i\not\in R\cup S$, $j\in R\cup S$ and $ b_{i}= b_{j}$.
Hence,  $j\in R$ because it is the second occurrence of the same number (since $i<j$) and thus is the second entry of its block. Conversely, whenever $ b_i= b_j$ with $i<j$ and $j\in R$, the pair $(i,j)$ is an inversion of $\alpha$ but not of $\beta$.  Therefore,
\begin{equation}\label{stats:inv}
\inv(\alpha)=\inv(\beta)+|R|.
\end{equation}
and combining~\eqref{stats:disp} and~\eqref{stats:inv} we obtain
\begin{equation} \label{stats}
q^{\disp(\alpha)}t^{\inv(\alpha)}=q^{\disp(\beta)+|R|+|S|}t^{\inv(\beta)+|R|}.
\end{equation}
Consequently,
\begin{align*}
\sum_{\alpha\in \eta^{-1}(\beta)}q^{\disp(\alpha)}t^{\inv(\alpha)}&=\sum_{R\subseteq \RRR}\sum_{S\subseteq \SSS}q^{\disp(\beta)+|R|+|S|}t^{\inv(\beta)+|R|}\\
&=q^{\disp(\beta)}t^{\inv(\beta)}\lp\sum_{R\subseteq \RRR}(qt)^{|R|}\rp \lp\sum_{S\subseteq \SSS}q^{|S|}\rp\\
&=q^{\disp(\beta)}t^{\inv(\beta)}(1+q)^{|\SSS(\beta)|} (1+qt)^{|\RRR(\beta)|},
\end{align*}
and summing over all $\beta\in\UPF_n$ yields the theorem.
\end{proof} 

\begin{example}
Let $\beta=6411624\in\UPF_7$, which has block structure $\beta^\uparrow=112|44|66$.  Positions 4,5,7 contain the second elements of the blocks.  Note that $\RRR(\beta)=\{7\}$, because:
\begin{itemize}
\item $4\notin\RRR(\beta)$ because the 1 in position 4 is the second element of the \textit{first} block.
\item $5\notin\RRR(\beta)$ because the 6 in position 5 is the second element of $\pi_3$, but it occurs to the left of the largest element of $\pi_2$, namely the 4 in position 7.
\item $7\in\RRR(\beta)$ because  the 4 in position 7 is the second element of $\pi_2$, and it occurs to the right of the largest element of $\pi_1$, namely the 2 in position 5.
\end{itemize}
Meanwhile, $\SSS(\beta)=\{6\}$.  We have
\[\spot_\beta=6412735,\quad \disp(\beta)=0+0+0+1+1+1+1=4,\quad \inv(\beta)=|\{12,13,14,16,17,23,24,26,56,57\}|=10.\]
The possibilities for $R$ and $S$ and the corresponding $2$-interval parking functions $\alpha=g_\beta(R,S)$ are as follows.
\[
\begin{array}{ccccc}
R&S&\alpha&\spot_\alpha&\disp(\alpha)\\ \hline
\emptyset&\emptyset	&6411624&6412735 & 0+0+0+1+1+1+1 = 4\\
7&\emptyset&6\underline{4}1162\underline{3}&6412735& 0+0+0+1+1+1+2 = 5\\
\emptyset&6&6411614&6412735& 0+0+0+1+1+2+1 = 5\\
7&6&6\underline{4}1161\underline{3}&6412735& 0+0+0+1+1+2+2 = 6
\end{array}\]
In the cases that $R$ is nonempty, the new inversion is underlined.
\end{example}

Since $\Phi_{n,2}(q,t)$ can be understood as enumerating $2$-interval parking functions by either displacement and inversions or by displacement and major index, it is natural to want a version of \cref{thm:disp-enumerator-for-2-interval} that involves $\maj(\beta)$ instead of $\inv(\beta)$. To obtain this from \cref{thm:disp-enumerator-for-2-interval} and \Cref{thm:main-Foata}, we must study what happens to $|\RRR(\beta)|$ and $|\SSS(\beta)|$ when we apply the Foata transfrom to a unit interval parking function $\beta$. It turns out that these do not change, as the following lemma will help us show.

\begin{lemma}\label{lem: Foata-adjacent-invariance}
Let $(a_1,a_2,\dots,a_n)\in \N^n$.  Let $j,k\in[n]$ such that $a_j\leq a_k$, and let $i\geq\max(j,k)$.
    \begin{enumerate}
    \item If $a_{i+1}\not\in \{a_j,\dots,a_k-1\}$, then $a_j$ and $a_k$ appear in the same order in $F(a_1 a_2\cdots a_{i+1})$ as they do in $F(a_1 a_2\cdots a_i)$.
     \item  If $\{a_j,\dots,a_k-1\}\cap \{a_{i+1},\dots,a_n\}=\emptyset$, then $a_j$ and $a_k$ appear in the same order in $F(a_1\cdots a_n)$ as they do in $F(a_1 a_2\cdots a_i)$.
    \end{enumerate}
\end{lemma}

\begin{proof}
(1) The condition $a_{i+1}\not\in \{a_j,\dots,a_k-1\}$ implies that, at the $(i+1)$st stage of the algorithm, separators are placed either after both~$a_j$ and~$a_k$, or after neither. 
If both, then $a_j$ and $a_k$ belong to different segments.  If neither, then if they are in the same segment, then neither one is the last entry in that segment.  Therefore, in all cases, the cycling step maintains the relative position of $a_j$ and $a_k$, so they appear in the same order in $F(a_1\cdots a_{i+1})$ as in $F(a_1\cdots a_i)$.

(2) This claim follows from applying (1) repeatedly.
\end{proof}

\begin{lemma}\label{lem:Foata-preserves-R-and-S}
    If $\alpha\in \UPF_n$, then $|\RRR(\alpha)|=|\RRR(F(\alpha))|$ and $|\SSS(\alpha)|=|\SSS(F(\alpha))|.$
\end{lemma}

\begin{proof}
Since $F$ preserves content, it preserves block structure, and so we have $|\SSS(\alpha)|=|\SSS(F(\alpha))|$ by~\eqref{S-formula}.

For $\alpha=(a_1,a_2,\dots,a_n) \in \UPF_n$, let $\PPP(\alpha)$ be the set of positions of the second entries of blocks.  Equivalently, by block structure,
\[\PPP(\alpha)=\{i\in[n]\suchthat a_j=a_i\text{ for some } j<i\}.\]
Note that $\RRR(\alpha)\subseteq \PPP(\alpha)$.  Moreover, we have $|\PPP(\alpha)|=|\PPP(F(\alpha))|$ because $F$ preserves content.
For convenience, define $\RRR'(\alpha)=\{a_i\suchthat i\in\RRR(\alpha)\}$, so that $|\RRR'(\alpha)|=|\RRR(\alpha)|$.

We claim that $\RRR'(\alpha)=\RRR'(F(\alpha))$.  Indeed, let $k\in\PPP(\alpha)$, and let $\pi_s$ be the block containing $a_k$.  Taking $a_j$ to be either the first entry of $\pi_s$ or the last entry of $\pi_{s-1}$, and taking $i=\max(j,k)$, block structure implies that the elements $a_j,a_k$ satisfy the conditions of \cref{lem: Foata-adjacent-invariance}, hence appear in the same order in $F(\alpha)$ as they do in $F(a_1 a_2\cdots a_i)$, hence the same order as in $a_1\cdots a_i$ (due to our choice of~$i$).  In particular, $a_k\in\RRR'(\alpha)$ if and only if $a_k\in\RRR'(F(\alpha))$.  Therefore,
\[|\RRR(\alpha)| = |\RRR'(\alpha)| = |\RRR'(F(\alpha))| = |\RRR(F(\alpha))|.\qedhere\]
\end{proof}

\begin{example}
Let $\alpha=34411\in\UPF_5$, with block structure $11\mid3\mid44$.  Then $F(\alpha)=13441$.  Note that
\begin{align*}
\PPP(\alpha)    &= \{3,5\}, & \RRR(\alpha)    &= \{3\},\\
\PPP(F(\alpha)) &= \{4,5\}, & \RRR(F(\alpha)) &= \{4\},
\end{align*}
and $\RRR'(\alpha)=\RRR'(F(\alpha))=\{4\}$.
\end{example}

Applying \Cref{lem:Foata-preserves-R-and-S} and \Cref{thm:main-Foata} to  \Cref{thm:disp-enumerator-for-2-interval}, we obtain a major-index enumeration for $2$-interval parking functions.
\begin{corollary}\label{cor:maj-version-2}
For all $n\geq 1$,
\[
\sum_{\alpha\in \IPF_n(2)}q^{\disp(\alpha)}t^{\maj(\alpha)}=\sum_{\beta\in \UPF_n} q^{\disp(\beta)}(1+q)^{|\SSS(\beta)|}(1+qt)^{|\RRR(\beta)|}t^{\maj(\beta)}.
\]
\end{corollary}


\section{Enumeration via ciphers} \label{sec:upf-inversions}
 For positive integer $n$ and nonnegative integer $k$, define 
\[\UPF_{n,k}^{\inv}=\{\alpha\in \UPF_n\suchthat \inv(\alpha)=k\},\quad \upf_{n,k}^{\inv}=|\UPF_{n,k}^{\inv}|.\]
Setting $q=1$ in \cref{thm:disp-inv-enumerator-for-upfs} and extracting the coefficient of $t^k$ yields
\begin{equation}\label{upf-inv-asc}
\upf_{n,k}^{\inv} = \sum_{\substack{\sigma \in \Sym_n\\ \inv(\sigma)=k}} 2^{\asc(\sigma)}
\end{equation}
although it is not clear how to convert this expression to a closed formula.  In this section, we describe how work of Avalos and Bly \cite{avalos_sequences_2020}, together with block structure theory (\cref{thm:upf_rearrangement}), can be used to determine the numbers $\upf_{n,k}^{\inv}$ through the use of \defterm{ciphers}, an analogue of Lehmer codes for unit interval parking functions.
We then use ciphers to provide a new proof of \cite[Theorem 1.2]{elder2024parking}, which counts Boolean intervals of $\S_n$ under the right weak order by certain unit interval parking functions.

\subsection{The Avalos--Bly and cipher bijections}
Let $m$ and $n$ be nonnegative integers.  
In what follows, $T=(T_1,T_2,\dots,T_m)$ 
always denotes an ordered list of multisets of nonnegative integers.  
Let $\type(T)=(|T_1|,|T_2|,\dots,|T_m|)$.
Following Avalos and Bly \cite[Defns.~1.8,~1.10]{avalos_sequences_2020}, define
\begin{align*}
\FF^m_n &= \{T=(T_1,T_2,\dots,T_m)\suchthat |T_1|+|T_2|+\cdots+|T_m|=n\text{ and } 0\leq x\leq \textstyle\sum_{j=1}^{i-1}|T_j|\text{ for all } x\in T_i\},\\
\FF^m_n(k) &= \{T=(T_1,T_2,\dots,T_m)\in\FF^m_n \suchthat \sum_i\sum_{x\in T_i}x = k\}.
\end{align*}
Avalos and Bly proved~\cite[Proposition 1.9]{avalos_sequences_2020} that there is a bijection
\[\varphi\colon[m]^n\to\FF^m_n\]
defined as follows.  Given $w=(w_1,w_2,\dots,w_n)\in[m]^n$ and $j\in[n]$, define
\begin{equation} \label{inv-sigma-j}
\inv(w,j)=|\{\ell\in[j+1,n]\suchthat w_j>w_\ell\}|
\end{equation}
(the number of inversions with $w_j$ as the left element).  Let $T_i = \{\inv(w,j) \suchthat w_j=i\}$ (as a multiset)
and set $\varphi(w)=(T_1,T_2,\dots,T_m)$.  In particular, since $\inv(w)=\sum_{j=1}^m\inv(w,j)$, it follows that $\varphi$ restricts to a bijection
\[\{w\in[m]^n\suchthat\inv(w)=k\} \to \FF^m_n(k),\]
which is~\cite[Proposition~1.11]{avalos_sequences_2020}.  Define the
\defterm{content} of $w\in[m]^n$ as $\cont(w)=(|w^{-1}(1)|,\dots,|w^{-1}(m)|)$ (regarding $w$ as a map $w:[n]\to [m]$). Then the Avalos--Bly construction implies that $\type(T)=\cont(w)$.  In particular, for each weak integer composition $c$ of $n$, the map $\varphi$ restricts to bijections
\begin{equation} \label{avalos-bly}
\begin{aligned}
\{w\in[m]^n \suchthat\cont(w)=c\} &\to \{T\in\FF^m_n \suchthat \type(T)=c\},\\
\{w\in[m]^n \suchthat\cont(w)=c,\ \inv(w)=k\} &\to  \{T\in\FF^m_n(k) \suchthat \type(T)=c\}.
\end{aligned}
\end{equation}

Combining the Avalos--Bly bijections~\eqref{avalos-bly} with that of \cref{cor:block-bijection}, we obtain bijections
\begin{align*}
\{\alpha\in\UPF_n\suchthat (|\pi_1|,\dots,|\pi_m|)=c\} &\to \{T\in\FF^m_n \suchthat \type(T)=c\},\\
\{\alpha\in\UPF_{n,k}^{\inv}\suchthat (|\pi_1|,\dots,|\pi_m|)=c\} &\to \{T\in\FF^m_n(k) \suchthat \type(T)=c\},
\end{align*}
for each (strict) integer composition $c$.  We can now state the main result of this section, which follows from gluing these last bijections together over all integer compositions $c$.

\begin{theorem}[Cipher Bijection]\label{thm:cipher-bijection}
Let $n$ and $m$ be positive integers and let $k$ be a nonnegative integer. 
Then
\begin{equation} \label{UPF-to-Gmnk}
\psi\colon\UPF_{n,k}^{\inv} \to \bigcup_{m=1}^n \GG^m_n(k),
\end{equation}
where
\[\GG^m_n(k) = \{(T_1,T_2,\dots,T_m)\in\FF^m_n(k)\suchthat |T_i|>0\mbox{ for all } i\}\]
is a bijection.
\end{theorem}

We call $\psi(\alpha)$ the \defterm{cipher} of $\alpha$.  We  represent ciphers by writing each $T_i$ as a string of boldface numbers in weakly decreasing order, separated by bars.  For instance,
$\pmb{00\big|110\big|31\big|1}$ is an element of $\GG_8^4(7)$.

\begin{example}[One inversion]
Let $n\geq2$.  We calculate $\upf_{n,1}^{\inv}$.
By~\eqref{UPF-to-Gmnk}, we have a bijection
\[
\UPF_{n,1}^{\inv}\to\bigcup_{m=1}^n \GG^m_n(1).
\]
The right-hand side consists of sequences $(T_1,T_2,\dots,T_m)$ where each $T_i$ is a nonempty multiset of 0's and 1's, containing a total of $n-1$ 0's and one $1$, with the 1 not appearing in $T_1$ (so in particular $m\geq2$).
The ciphers for such a sequence are precisely those words that start with $\pmb{0}$, followed by some permutation of the following substrings: one copy of $\pmb{\big|1}$, and $n-2$ more strings each of which can be either $\pmb{0}$ or $\pmb{\big|0}$.
The number of such ciphers is thus \begin{equation} \label{upf:n1}
\upf_{n,1}^{\inv}=2^{n-2}(n-1).
\end{equation}
\end{example}

\begin{example}[Two inversions]
Let $n\geq3$.  We calculate $\upf_{n,2}^{\inv}$.  Here are the possibilities:
\begin{enumerate}
\item The cipher has one $\pmb{2}$ and $n-1$ $\pmb{0}$'s.  Since the $2$ cannot appear before the second $0$, 
 there are $2^{n-2}(n-2)$ such words.
\item The cipher has two $\pmb{1}$'s and $n-2$ $\pmb{0}$'s.  There are two subcases:
\begin{enumerate}
\item The $\pmb{1}$'s are in different blocks.
\item[] Count: $2^{n-3}\binom{n-1}{2}$
\item The $\pmb{1}$'s are in the same block.
\item[] Count: $2^{n-3}(n-2)$
\end{enumerate}
\end{enumerate}
Adding these counts, we obtain
\begin{equation} \label{upf:n2}
\upf_{n,2}^{\inv} = 2^{n-3}\left(3(n-2)+\binom{n-1}{2}\right)=2^{n-4}(n-2)(n+5).
\end{equation}
\end{example}

\begin{example}[Three inversions] Let $n\geq 3$.
We calculate $\upf_{n,3}^{\inv}$.  The descriptions of cases and their associated counts are as follows:
\[
\begin{array}{ll}
\text{There is one $\pmb{3}$}
    & 2^2 \cdot 2^{n-4} \cdot (n-3)
\\ [4pt]
\text{There is one $\pmb{2}$ and one $\pmb{1}$ and they are in the same block}
    & 2\cdot2^{n-4} \cdot (n-3)
\\ [4pt]
\text{There is a $\pmb{1}$ in the second block and a $\pmb{2}$ in a different block}
    & 2^{n-3} \cdot (n-2)
\\ [4pt]
\text{There is a $\pmb{1}$ not in the second block and a $\pmb{2}$ in a different block}
    & 2\cdot 2^{n-4} \cdot (n-3)(n-2)
\\ [4pt]
\text{There are three $\pmb{1}$'s, all in different blocks}
    & 2^{n-4} \cdot \binom{n-1}{3}
\\ [4pt]
\text{There are three $\pmb{1}$'s, and two of them are in a block together}
    & 2^{n-4} \cdot (n-3)(n-2)
\\ [4pt]
\text{There are three $\pmb{1}$'s, and they are all in the same block}
    & 2^{n-4} \cdot (n-3)\\ [4pt]
\end{array}
\]
Adding up the counts, we obtain
\begin{equation} \label{upf:n3}
\upf_{n,3}^{\inv}=2^{n-4}\lp \frac{1}{6}n^3 + 2n^2-\frac{25}{6}n-8\rp.
\end{equation}
\end{example}

Comparing the formulas~\eqref{upf:n1}, \eqref{upf:n2}, and \eqref{upf:n3}, we are led to the following general conjecture, which we leave as an open problem.

\begin{conjecture} \label{conj:count-upfnk}
Let $n\geq 1$ and $k\in \{0,1,\dots,n(n-1)/2\}$.  Then
$\upf_{n,k}^{\inv}=O(2^{n-2k} n^k)$.
\end{conjecture}

\subsection{Ciphers and Lehmer codes}
We describe another way to construct ciphers by way of Lehmer codes for permutations. Let $E_{n,k}$ be the set of words $w\in \{0,1,\dots,n-1\}^n$  satisfying 
\begin{enumerate}
    \item $w_i\leq i-1$\quad \text{for all $i\in [n]$, and }
    \item $\sum_{i=1}^n w_i=k$.
\end{enumerate}
Then a cipher $T\in \mathcal{G}_n(k)$ can be constructed from a word $w\in E_{n,k}$ by inserting bars.  If $w_i<w_{i+1}$, we \textit{must} insert a bar between $w_i$ and $w_{i+1}$; otherwise, we \textit{may} insert a bar.
Hence, we can construct $2^{n-1-\asc(w)}$ ciphers from $w$, and every cipher arises in this way. It follows that
\begin{equation} \label{upfnk-from-words}
\upf_{n,k}^{\inv}=\sum_{w\in E_{n,k}}2^{n-1-\asc(w)}.
\end{equation}

Define 
\[E_n=\bigcup_{k\geq 0} E_{n,k} = \{0\}\times \{0,1\}\times\cdots \times \{0,1,\dots,n-1\}.\]
Then $E_n$ is the set of \defterm{Lehmer codes} (backwards); see \cite[p.~34]{stanley2012enumerative}. There is a bijection $\gamma:S_n\to E_n$ from  permutations to Lehmer codes: given $\sigma\in S_n$, set $\gamma(\sigma)_{\sigma(i)}=\#\{\ell\in [i-1]:(i,\ell)\in \Inv(\sigma)\}$ 
It is not hard to see that for all $\sigma\in S_n$,
\begin{align} \label{lehmer:1}
\Des(\sigma^{-1})&=\Asc(\gamma(\sigma))
\intertext{and}\label{lehmer:2}
\sum_{i=1}^n \gamma(\sigma)_{i} &=\inv(\sigma).
\end{align}

Applying~\eqref{lehmer:1} to~\eqref{upfnk-from-words} and observing that  $\asc(\sigma)+\des(\sigma)=n-1$ for all $\sigma\in \S_n$, we obtain
\[
\upf_{n,k}^{\inv}=\sum_{\substack{\sigma\in \S_n\\\inv(\sigma)=k}}2^{n-\des(\sigma^{-1})-1}=\sum_{\substack{\sigma\in \S_n\\\inv(\sigma)=k}}2^{\asc(\sigma)},
\]
which recovers the $q=1$ specialization of \cref{thm:disp-inv-enumerator-for-upfs}. 

Define the \defterm{underlying code} of a cipher $T\in\mathcal{G}_n(k)$ to be the word $\code(T)\in E_n$ obtained from~$T$ by removing all bars.  Likewise, the underlying code of a UPF $\alpha$ is $\code(\alpha)=\code(\psi(\alpha))$, where $\psi$ is the Avalos--Bly bijection.

\begin{proposition}
    For a unit interval parking function $\alpha$, the Lehmer code of $\spot_\alpha$ is the underlying code of $\alpha$. In particular, if $\beta$ is a unit interval parking function, then $\alpha$ and $\beta$ have the same spot permutation if and only if $\psi(\alpha)$ and $\psi(\beta)$ have the same underlying code.
\end{proposition}

\begin{proof}
Let $\pi_1\mid \pi_2\mid\cdots\mid\pi_m$ be the block structure of $\alpha=(a_1,a_2,\ldots,a_n)$. 
Following \cref{cor:block-bijection}, we get a surjective function $s:[n]\to [m]$ from $\alpha$ by letting $s(i)$ be the index of the block to which $a_i$ belongs.

As in the Avalos--Bly construction described above, for each $i\in[m]$, let $T_i$ denote the multiset of all numbers $\inv(s,j)$ (see~\eqref{inv-sigma-j}) such that $s(j)=i$. Note that $\psi(\alpha)=(T_1,\dots,T_m)$ and so $w:=\code(\psi(\alpha))$ is the concatenation $T_1T_2\cdots T_m$, where each $T_i$ is sorted in weakly decreasing order. 

Let $t_k=|T_1|+\cdots+|T_{k-1}|$ and 
set $j_r$ to be the index of the $r$th element of $s$ equal to $k$. 
For $1\leq r \leq |\pi_{k}|$, define $T_{k}^r=\inv(s,j_r)=|\{\ell: (j_r,\ell)\in \Inv(s)\}| $ so that $T_{k}$ is the concatenation $T_k^1T_k^2\cdots T_k^{|\pi_k|}$. Note that $\{\ell: (j_{r_1},\ell)\in \Inv(s)\}\supseteq \{\ell: (j_{r_2},\ell)\in \Inv(s)\}$ for $r_1\leq r_2$, so that concatenation is in weakly decreasing order. In particular $w_{t_k+r}=T_k^r$

Since $\alpha$ has the same inversions as both $s$ and $\spot_\alpha$, we have $T_{k}^r=|\{\ell:(j_r,\ell)\in \Inv(\spot_\alpha)\}|$.  
Since $\alpha_{j_r}$ is the $r$th element of the $k$st  block, it will park in spot $t_k+r$. In particular, $\spot_\alpha(j_r)=t_k+r$. Thus,$\gamma(\spot_\alpha)_{t_k+r}=|\{\ell:(j_r,\ell)\in \Inv(\spot_\alpha)\}|=w_{t_k+r}$. That is, the underlying code of $\alpha$ is the Lehmer code of $\spot_\alpha$.

\end{proof}

\begin{example} \label{exa:cipher-from-w}
Let $w=00110311\in E_{8,7}$.  The corresponding ciphers are of the form $0\cdot0|1\cdot1\cdot0|3\cdot1\cdot1$, where each dot is replaced either with $|$ or with the empty symbol. The unit interval parking function corresponding to any of these 32 ciphers has 
$\spot_\alpha=13642785$, e.g.,
\[\begin{array}{ccccc}
00|110|311 &\xleftarrow{\psi}&  13631674
&\xrightarrow{\spot}& 13642785\\
00|1|10|3|11 &\xleftarrow{\psi}& 13641774 &\xrightarrow{\spot}& 13642785.
\end{array}\]

\end{example}

\subsection{Unit Fubini rankings}

A \defterm{unit Fubini ranking} is a unit interval parking function whose blocks all have size 1 or 2. Unit Fubini rankings were introduced in \cite{elder2024parking} to study Boolean intervals of $\S_n$ under the right weak order. We can use ciphers to provide a new proof of the main result of that paper.

\begin{theorem}[{\cite[Theorem 1.2]{elder2024parking}}]\label{count-unit-Fubini}
Unit Fubini rankings with $n-r$ blocks and $k$ inversions are in bijection with rank-$r$ Boolean intervals $[\sigma,\tau]$ of $\S_n$ under the right weak order with $\inv(\sigma)=k$.
\end{theorem}

This statement is ostensibly a refinement of~\cite[Theorem~1.2]{elder2024parking}, which does not mention inversions explicitly, but it is not difficult to use the methods of~\cite{elder2024parking} together with elementary facts such as \cref{lem: inv of outcome} to obtain the more refined version.

\begin{proof}
Call a set of integers \defterm{sparse} if it contains no two consecutive integers.
Let $\UFR_{n,k}^m$ be the set of unit Fubini rankings of length $n$ with $m$ blocks and $k$ inversions.  The bijection \eqref{UPF-to-Gmnk}
restricts to a bijection
\[\psi\colon\UFR_{n,k}^m \to \hat\FF^m_n(k)\coloneqq\{(T_1,T_2,\dots,T_m)\in\FF^m_n(k)\suchthat |T_i|\in\{1,2\}\ \ \mbox{for all }i\}.\]
As in \cref{exa:cipher-from-w}, the elements of $\FF^m_n(k)$ can be represented as ciphers as follows: choose a word $w\in E_{n,k}$, insert bars at all its ascents, and insert additional bars as needed so as to obtain $m$ blocks, all of size 1 or 2.  That means that we need a total of $m-1$ bars; if $B$ is the set of non-ascents where there are bars, then $|B|=m-1-|\Asc(w)|$.  Thus,we have a bijection
\[\hat\FF^m_n(k) \to \left\{(w,B)\suchthat
\begin{array}{l}
w\in E_{n,k},\\
B\subseteq [n-1]\sm\Asc(w),\\
|B|=m-1-|\Asc(w)|, \mbox{ and}\\
{[n-1]}\sm(\Asc(w)\cup B) \text{ is sparse}
\end{array}\right\}.\]
We may replace $(w,B)$ with the equivalent data $(\sigma,C)$, where $\sigma=\gamma(w)$ and $C=\Asc(\sigma)\sm B$.
(By~\eqref{lehmer:1}, we have $[n-1]\sm\Asc(w)=[n-1]\sm\Des(\sigma)=\Asc(\sigma)$.)  Translating the conditions on $(w,B)$ into conditions on $(\sigma,C)$, we obtain a bijection
\[\hat\FF^m_n(k) \to \left\{(\sigma,C)\suchthat
\begin{array}{l}
\sigma\in\Sym_n, \ \inv(\sigma)=k,\\
C\subseteq\Asc(\sigma), \ |C|=n-m,\mbox{ and }\\
C \text{ is sparse}
\end{array}\right\}.\]
On the other hand, a Boolean interval in the right weak order with bottom element $\sigma$ and rank $r$ is given precisely by choosing $r$ pairwise commuting elementary transpositions $s_i$ such that $\inv(\sigma s_i)>\inv(\sigma)$ for every $i$, which is equivalent to choosing a sparse subset of $\Asc(\sigma)$ of size $r$.
\end{proof}

 \begin{remark}\label{remark: fib-count}
Let $\Fib_i$ denote the $i$th Fibonacci number: $(\Fib_0,\Fib_1,\Fib_2,\Fib_3,\dots)=(1,1,2,3,\dots)$.
Theorem~1.1 of~\cite{elder2024parking} states that for each $\sigma\in\Sym_n$, the number of Boolean intervals in right weak order with bottom element $\sigma$ is
$\prod_{i=1}^k \Fib_{a_i(\sigma)}$,
where $a_1(\sigma),\dots,a_k(\sigma)$ are the lengths of the maximal ascending subwords of $\sigma=(\sigma_1,\dots,\sigma_n)$, or, equivalently, of the maximal non-ascending subwords of $w=\gamma^{-1}(\sigma^{-1})$.  The proof of Theorem~\ref{count-unit-Fubini} provides a generating function for Boolean intervals by rank, as follows.  Define the \defterm{Fibonacci polynomials} $\Fib_n(q)$ by
\[\Fib_n(q)=\sum_{\substack{S\subseteq [n-1]\\\text{$S$ sparse}}}q^{|S|}\]
(see sequence~\#\seqnum{A011973} in \cite{OEIS}), so that $\Fib_n(1)=\Fib_n$.  Then
\begin{equation}
\sum_{\substack{\tau\in\Sym_n\\\ [\sigma,\tau]\ \text{Boolean}}} q^{\inv(\tau)-\inv(\sigma)} = \prod_{i=1}^k \Fib_{a_i(\sigma)}(q).
\end{equation}
\end{remark}

\section{Further Problems}\label{sec:future-work}

We provide some open problems for further study.
\begin{problem}
Describe a block structure for $\ell$-interval parking functions that generalizes the $\ell=1$ case (\cref{thm:upf_rearrangement}).
\end{problem} 

As we noted in \Cref{sec:block structure}, 
the block structure of unit interval parking functions appears to be unique to the case $\ell=1$.  For $\ell>1$, the problem seems to be much harder.

\begin{problem}
What can be said about the effect of the Foata transform on the \textit{outcome} of a parking function?
\end{problem}

One obvious thing is that if $\alpha\in \S_n$, then $\spot_{F(\alpha)}=F(\spot_\alpha)$ (because $\alpha=\spot_\alpha$). However, this does not hold in general. For instance, let $\alpha=121$. Then $F(\alpha)=211$, but $\spot_{F(\alpha)}=213$ whereas $F(\spot_\alpha)=F(123)=123.$ For unit interval parking functions (which the previous $\alpha$ is not), it seems that $F$ acts on the spots permutations as it did for permutations:
\begin{conjecture} \label{conj:1}
    For any unit interval parking function $\alpha$, we have $\spot_{F(\alpha)}=F(\spot_\alpha)$.
\end{conjecture}
We have verified \cref{conj:1} for all unit interval parking functions of length at most $8$.

\begin{problem}
    Count $\ell$-interval parking functions in terms of $(\ell\pm 1)$-interval parking functions.
\end{problem}
Theorems \ref{thm:disp-inv-enumerator-for-upfs} and \ref{thm:disp-enumerator-for-2-interval} as well as \Cref{ell-minus-two} are solutions to the above problem for $\ell=1,2,n-2$, respectively. Let $\IPF_{n}(\ell)^{\max}$ be the set of parking functions $\alpha$ of length $n$ such that $\maxdisp(\alpha)=\ell$. It follows from our enumerative results that for all $n\geq 1$
\[n!=|\IPF_n(0)^{\max}|\leq|\IPF_n(1)^{\max}|\leq |\IPF_n(2)^{\max}|\quad \text{and}\quad |\IPF_n(n-2)^{\max}|\geq |\IPF_n(n-1)^{\max}|=n^{n-2} \]
Experimentally, it seems this pattern continues. See \Cref{tab:maxdisp}.
\begin{conjecture}
    For each $n\geq 1$, the sequence $(|\IPF_{n}(\ell)^{\max}|)_{\ell=0}^{n-1}$ is unimodal i.e. if $a_\ell=|\IPF_{n}(\ell)^{\max}|$, there exists an integer $c\in \{0,1,\dots,n-1\}$ such that
    \[n!=a_0\leq a_1\leq a_2 \leq \cdots\leq a_c \geq\cdots  \geq  a_{n-2}\geq a_{n-1}=n^{n-2}. \]
\end{conjecture}

\begin{table}[ht]
    \centering
    \begin{tabular}{|c||l|l|l|l|l|l|l|l|l|}\hline
         \diagbox{$n$}{$\ell$}&0&1&2&3&4&5&6 &7&8\\\hline\hline
         1&1&&&&&&&&\\\hline
2&2 &1&&&&&&&\\\hline
3&6 &7 &3&&&&&&\\\hline
4&24 &51 &34 &16&&&&&\\\hline
5&120& 421& 377& 253& 125&&&&\\\hline
6&720 &3963& 4594& 3688& 2546& 1296&&&\\\hline
7&5040& 42253& 62145& 57398& 46142 &32359&16807&&\\\hline
8&40320& 505515& 929856& 979430& 865970& 702292 &497442& 262144&\\\hline
9&362880& 6724381& 15298809& 18289811 &17520519& 15455851& 12587507& 8977273& 4782969\\\hline
    \end{tabular}
    \caption{$|\IPF_{n}(\ell)^{\max}|$ up to $n=9$.}
    \label{tab:maxdisp}
\end{table}

\section*{Acknowledgments}
The genesis for this project was the Graduate Research Workshop in Combinatorics 2024, hosted by University of Wisconsin, Milwaukee, which was supported in part by NSF Grant DMS~--~1953445. The sixth author was partially supported by the National Science Foundation Graduate Research Fellowship under Grant  DGE~--~2139899.
We thank the developers of 
SageMath~\cite{sage}, which was useful in this research, and the CoCalc~\cite{SMC} collaboration platform.
We also thank Steve Butler, Ari Cruz, Kim Harry,  Matt McClinton, Keith Sullivan, and Mei Yin for their comments and guidance at the start of this project.
This work was supported
by a grant from the Simons Foundation (Travel Support for Mathematicians, P. E. Harris).

\bibliographystyle{plain}
\bibliography{Bibliography.bib}
\end{document}